\documentclass[12pt]{article}
\usepackage{amssymb,amsmath,amsfonts,mathrsfs,amsthm,epsfig,latexsym,color}
\usepackage{enumitem,geometry}
\usepackage{fourier}
\usepackage[all]{xy}
\usepackage{setspace}
\usepackage{indentfirst, float, subcaption}
\makeatletter
\newcommand{\subsectionruninhead}{\@startsection{subsection}{2}{0mm}
{-\baselineskip}{-0mm}{\bf\large}}
\newcommand{\subsubsectionruninhead}{\@startsection{subsubsection}{3}{0mm}
{-\baselineskip}{-0mm}{\bf\normalsize}}
\makeatother

\newtheorem*{theorem*}{Theorem}
\newtheorem*{proof*}{Proof}
\newtheorem*{proposition*}{Proposition}
\newtheorem*{corollary*}{Corollary}
\newtheorem*{claim*}{Claim}
\newtheorem*{remark*}{Remark}
\newtheorem*{problem*}{Problem}
\newtheorem{theorem}{Theorem}[section]
\newtheorem{proposition}{Proposition}[section]

\newtheorem{corollary}[proposition]{Corollary}
\newtheorem{lemma}[proposition]{Lemma}

\theoremstyle{definition}
\newtheorem{definition}[proposition]{Definition}
\theoremstyle{remark}
\newtheorem{remark}[proposition]{Remark}

\numberwithin{equation}{section}

\def\NN{\mathbb{N}}
\def\RR{\mathbb{R}}

\def\TT{\mathbb{T}}
\def\ZZ{\mathbb{Z}}

\def\d0{\delta^{(0)}}

\begin{document}
\title{Derived-from-expanding endomorphism on $\mathbb{T}^2$}
\author{Daohua Yu}
\date{\today}
\maketitle
\begin{abstract}
Assume that $f$ is a $C^r(r\geq 3)$ specially partially hyperbolic endomorphism on the 2-torus which is homotopic to an expanding linear endomorphism $A$ with irrational eigenvalues. We prove that $f$ and $A$ are topologically conjugate, if and only if $f$ is area-expanding. If $f$ is area-expanding and the center bundle is $C^1$, then the topological conjugacy between $f$ and $A$ is $C^{\max\{r-3,1\}+\alpha}$. In particular, if $r=\omega$, the conjugacy is $C^{\omega}$.
\end{abstract}
\section{Introduction}
\indent  In 1999, Diaz, Pujals, and Ures \cite{DPU} proved that robustly transitive diffeomorphisms on three-dimensional manifolds are partially hyperbolic. Important examples of robustly transitive partially hyperbolic diffeomorphisms, which are given by Shub \cite{Shub2} and Mane \cite{Mane} in the 70s, are all derived-from-Anosov (i.e. homotopic to Anosov) diffeomorphisms.
However, the classification of all robustly transitive derived-from-Anosov diffeomorphisms remains an open problem, even on the 3-torus. In order to address the problem, Pujals suggested that it might be a good choice to classify derived-from-expanding (i.e. homotopic to expanding) maps that are robustly transitive on the 2-torus \cite{Andersson, AnderssonRanter}.\\
\indent For any continuous map $f$ from the 2-torus to itself, there is a unique linear toral endomorphism $A$ which is homotopic to $f$. If $A$ is expanding, then there is a unique semi-conjugacy between $f$ and $A$. In 1969, Shub \cite{Shub} proved that if $f$ is expanding, then $f$ is topologically conjugate to $A$.
Moreover, any expanding endomorphism is topologically conjugate to an affine expanding endomorphism of an infranilmanifold \cite{Shub, Franks, Gromov}. \\
\indent We decide to classify derived-from-expanding partially hyperbolic endomorphisms based on topological or smooth conjugacy.

\begin{definition}\cite{AnderssonRanter}
Let $f:\TT^2\to \TT^2$ be a $C^1$ endomorphism, i.e., a local diffeomorphism on $\TT^2$.
If there is a $Df-$invariant unstable closed convex cone $\mathscr{C}_x\subset T_x\mathbb{T}^2$ for all $x\in \mathbb{T}^2$, i.e.,
$$Df(x)\mathscr{C}_x\subset {\rm int}\mathscr{C}_{fx}, \qquad |Df(x)v|>1,\ \text{for all}\, v\in \mathscr{C}_x\ \text{with}\ |v|=1,$$
then $f$ is called a \emph{partially hyperbolic endomorphism}.
\end{definition}
\begin{remark}
Without loss of generality, we use one time iteration of $f$ instead of $l$ times iterations of $f$ in \cite{AnderssonRanter}. This kind of treatment is also used in \cite{Heshiwang}.
\end{remark}

\begin{problem*}
Assume that $f$ is a partially hyperbolic endomorphism on $\TT^2$, which is homotopic to an expanding linear map $A$.
\begin{enumerate}
\item  It is well known that there is a topological semi-conjugacy $h$, such that $hf=Ah$. Under what kind of condition, $f$ will be topologically conjugate to $A$?
\item Moreover, what can be said about the smoothness of $h$?
\end{enumerate}
\end{problem*}

In the following, we give some related concepts of partially hyperbolic endomorphisms. Later, we give partial results of the question.

According to \cite{HallHammerlindl}, for a partially hyperbolic endomorphism $f$ on $\mathbb{T}^2$, there is a one dimensional $Df-$invariant line bundle $$E^c_x=\bigcap_{n=0}^{+\infty}(Df^n(x))^{-1}(\mathscr{C}_{f^nx})^c,$$
which is called the center bundle.

\begin{definition}
Let $f$ be a partially hyperbolic endomorphism on $\mathbb{T}^2$.
If there is a $Df-$invariant unstable bundle $E^u$ such that for all $\{x_n\}_{n\in\ZZ}$ such that $f(x_n)=x_{n+1}$ and $x_0=x$,
$$E^u_x=\bigcap_{n=0}^{+\infty}Df^n(x_{-n})\mathscr{C}_{x_{-n}},$$
i.e., the unstable bundle is independent of the choice of backward orbits, then $f$ is called \emph{specially} partially hyperbolic endomorphism.
\end{definition}
\indent  Assume that $f$ is a specially partially hyperbolic endomorphism, which is homotopic to an expanding linear map $A$. In 2024, Andersson and Ranter \cite{AnderssonRanter} proved that when $A$ has integer eigenvalues, $f$ is topologically conjugate to $A$ if and only if $f$ is transitive. Additionally, they showed that if $f$ is strong volume expanding, then $f$ is transitive. (Strong volume expanding, i.e., the Jacobian of $f$ at every point is larger than the spectral radius of $A$.) However, their work does not address the case, where $A$ has irrational eigenvalues. \\
\indent In the irrational case, we prove that for a specially partially hyperbolic $f$, $f$ and $A$ are topologically conjugated, if and only if $f$ is area-expanding. Specifically, we can prove Theorem \ref{conjugacy}.
\begin{theorem}\label{conjugacy}
Assume that $A$ is an expanding linear endomorphism on the 2-torus,
with irrational eigenvalues $|\lambda^u_A|>|\lambda^c_A|>1$, $f$ is a $C^r(r\geq 3)$ specially partially hyperbolic endomorphism which is homotopic to $A$. Then the semi-conjugacy $h$ between $f$ and $A$ is a topological conjugacy, if and only if $f$ is area expanding, i.e.,
\begin{align}\label{*}
\exists \rho>1, \exists C>0, s.t., \forall x\in\TT^2, \forall n\in\NN,\quad
|\det(Df^n(x))| > C\rho^n.
\end{align}
\end{theorem}
After building topologically conjugacy between $f$ and $A$, we begin to investigate the smoothness of the conjugacy.
Below, we briefly recall some earlier results concerning the smooth conjugacy.\\
\indent In the history, the same periodic data of two smooth maps may imply that they are smooth conjugate, where periodic data refers to the Lyapunov exponents at corresponding periodic points. For $C^{k+\alpha}(k\geq 1)$ expanding maps on a circle, it is well known that maps with the same degree are topologically conjugate. Furthermore, if their periodic data are the same, the conjugacy between them is $C^k$ \cite{Arteaga}.
For $C^k(k\geq 2)$ Anosov diffeomorphisms on a two-dimensional manifold, de la Lave \cite{de la Llave} proved in 1992 that if their periodic data are the same then the conjugacy is $C^{k-\epsilon}$.
In 2008, Gogolev \cite{G08} proved that if their periodic data are the same, then a $C^2$ Anosov diffeomorphism on the 3-torus, lying within a $C^1$-neighborhood of a linear hyperbolic automorphism with real eigenvalues, is $C^{1+\alpha}$ conjugated with the linear one. \\
\indent For partially hyperbolic diffeomorphisms, su-integrable may imply the same periodic data in the center direction. In 2020 Gan and Shi \cite{Ganshi} studied a $C^{1+\alpha}$ partially hyperbolic diffeomorphism
$f$ which is homotopic to an Anosov automorphism $A$ on the 3-torus. They proved that if $f$ is su-integrable and topologically conjugate to $A$, then $f$ and $A$ have the same periodic data in the center direction, and the topological conjugacy between them is differentiable along the center foliation.\\
\indent For partially hyperbolic endomorphisms, similarly to su-integrability for partially hyperbolic diffeomorphisms, being "special"  may imply differentiability of the conjugacy in the stable direction .
Recent results about the smooth conjugacy confirm this principle, including Anosov endomorphisms \cite{AGGS, M, MR}, Anosov maps \cite{GS}, and derived from non-expanding Anosov endomorphisms \cite{GX}.
In such cases, $f$ and its linear part have the same stable periodic data, thus the conjugacy is smooth along the stable foliation.\\
\indent In the present paper, we investigate the problem of the smooth conjugacy for a $C^r$ specially partially hyperbolic endomorphism $f$, which is homotopic to an expanding linear endomorphism $A$ with irrational eigenvalues. This problem has not been previously studied by \cite{AGGS, M, MR, GS, GX}. Specifically, we can prove that if $f$ is area expanding and the center bundle is $C^1$, then the conjugacy is $C^{\max\{r-3,1\}+\alpha}$, i.e., Theorem \ref{differentiability}.
\begin{theorem}\label{differentiability}
Assume that $A$ is an expanding linear endomorphism on the 2-torus with irrational eigenvalues $|\lambda^u_A|>|\lambda^c_A|>1$, and $f$ is a $C^r(r\geq 3)$ specially partially hyperbolic endomorphism which is homotopic to $A$. If $f$ is area expanding, and $E^c$ is $C^1$, then
\begin{enumerate}
\item[(1)]
the semi-conjugacy $h$ between $f$ and $A$ is a $C^{1+\alpha}$ conjugacy for any $\alpha\in(0,1)$.
\item[(2)]
$h$ is $C^{r-3+\alpha}$, when $r\geq 5$; $h$ is $C^{\infty}$, when $r=\infty$.
\item[(3)]
$h$ is $C^{\omega}$, when $r=\omega$.
\end{enumerate}
\end{theorem}

\begin{remark}
When the semi-conjugacy $h$ between $f$ and $A$ is a $C^2$ conjugacy, the invariant center bundle $E^c$ of $f$ can be seen as the image of the invariant center bundle of $A$ under $Dh^{-1}$, so $E^c$
is $C^1$. So the added condition that the invariant center bundle $E^c$ is $C^1$ is reasonable.
\end{remark}

We conclude by stating the main theorem of the present paper as follows, which contains both Theorem \ref{conjugacy} and Theorem \ref{differentiability}.
\begin{theorem}[Main Theorem]\label{main}
Assume that $A$ is an expanding linear endomorphism on the 2-torus with irrational eigenvalues $|\lambda^u_A|>|\lambda^c_A|>1$, $f$ is a $C^r(r\geq 3)$ specially partially hyperbolic endomorphism which is homotopic to $A$.
\begin{enumerate}
\item[(1)]
the semi-conjugacy $h$ between $f$ and $A$ is a topological conjugacy
if and only if $f$ is area expanding.
\item[(2)]
If $f$ is area expanding and $E^c$ is $C^1$, then
$h$ is a $C^{1+\alpha}$ conjugacy for any $\alpha\in(0,1)$.
\item[(3)]
If $f$ is area expanding and $E^c$ is $C^1$, then
$h$ is $C^{r-3+\alpha}$, when $r\geq 5$; $h$ is $C^{\infty}$, when $r=\infty$.
\item[(4)]
If $f$ is area expanding and $E^c$ is $C^1$, then
$h$ is $C^{\omega}$, when $r=\omega$.
\end{enumerate}
\end{theorem}
The proof of item (1) is given in section 3.1, the proofs of item (2)-(3) are given in section 3.2 and the proof of item (4) is given in section 3.3.\\
\noindent\textbf{Acknowledgement:}
The author of this article is grateful for the valuable communication and suggestions from Shaobo Gan. This work is partially supported by National Key R\&D Program of China 2022YFA1005801 and NSFC 12161141002.
\section{Preliminary}
For a specially partially hyperbolic endomorphism $f$ on $\mathbb{T}^2$, there exists the decomposition $T\mathbb{T}^2=E^c\oplus E^u$, which satisfies for all $x\in\TT^2$,
$$Df(x)E^c_x=E^c_{fx},\quad Df(x)E^u_x=E^u_{fx},\quad ||Df|_{E^c_x}||<||Df|_{E^u_x}||,\quad ||Df|_{E^u_x}||>1.$$
The unstable bundle $E^u$ of $f$ is locally uniquely integrable, and the unstable manifolds of $f$ form an $f-$invariant foliation $\mathscr{F}^u$. \\
\indent Now further assume that $f$ is $C^r(r\geq 3)$ and homotopic to an expanding linear endomorphism $A$ with the irrational eigenvalues $|\lambda^u_A|>|\lambda^c_A|>1$.

\begin{lemma}\label{no periodic annulus}
Under the assumption of Theorem \ref{main}, there is no periodic annulus of $f$ on $\mathbb{T}^2$.
\end{lemma}
\begin{proof}
Suppose that there is a periodic annulus $\mathbb{A}$, i.e., $f^n\mathbb{A}=\mathbb{A}, n\in\NN$. Choose a loop $\gamma$ in $\mathbb{A}$ such that $[\gamma]$ represents the unit in $\pi_1(\mathbb{A})=\ZZ$. Since $f^n\gamma\subset f^n\mathbb{A}=\mathbb{A}$, there is an integer $m$ such that $[f^n\gamma]=m[\gamma]$. Notice that $[f^n\gamma]=A^n[\gamma]$ in $\pi_1(\mathbb{T}^2)=\ZZ^2$, we have that $[\gamma]$ is an eigenvector of $A$ in $\ZZ^2$. It contradicts the assumption that the eigenvalues of $A$ are all irrational.
\end{proof}

Applying Lemma \ref{no periodic annulus}, we can know that $f$ is dynamical coherent, and the lifts of center foliation and unstable foliation are quasi-isometric.

\begin{theorem}
\cite[Theorem B]{coherence}
Let $f:\TT^2\to \TT^2$ be partially hyperbolic. If $f$ does not admit a periodic centre annulus, then $f$ is dynamically coherent and leaf conjugate to its linearisation.
\end{theorem}

\begin{corollary}
Under the assumption of Theorem \ref{main}, $f$ is dynamical coherent.
\end{corollary}

\begin{lemma}\label{quasi-iso}
\cite{Heshiwang}
For any specially partially hyperbolic endomorphism $f$  on $\TT^2$ and any $f$-invariant foliation $\mathscr{F}$, if $\mathscr{F}$ contains a Reeb component on an annulus $\mathbb{A}\subset\TT^2$, then there exists
a positive integer $n$ such that $\mathbb{A}$ is completely $f$ $n$-invariant: $f^n(\mathbb{A})=\mathbb{A}, f^{-n}(\mathbb{A})=\mathbb{A}$.
\end{lemma}

\begin{corollary}
Under the assumption of Theorem \ref{main}, neither center foliation nor unstable foliation has reeb component, thus they are both suspensions of circle homeomorphisms. So they are orientable and the lifts of them are quasi-isometric.
\end{corollary}

\indent Denote by $\tilde{f}$ a lift of $f$ on $\RR^2$.
Since $f$ is a nonsingular endomorphism, its lift $\tilde{f}$ is a diffeomorphism on $\RR^2$ \cite{ManePugh}. Since $f$ and $A$ are homotopic and $A$ is expanding, there is a unique $\tilde{h}:\RR^2\to \RR^2$ such that $\tilde{h}\tilde{f}=A\tilde{h}$ and $\tilde{h}(x+n)=\tilde{h}(x)+n$ for any $n\in\ZZ^2$ \cite[Theorem 2]{Shub}.
Then $\tilde{h}$ can descend to $h:\mathbb{T}^2\to \mathbb{T}^2$, and $h$ satisfies
$hf=Ah$.\\
\indent Denote by $\mathscr{F}^{\sigma}$ the $f$-invariant foliation tangent to $f$-invariant bundle $E^{\sigma}_f$, and $\tilde{\mathscr{F}}^{\sigma}$ the lift of $\mathscr{F}^{\sigma}$, $\sigma = c,u$.\\
\indent Denote by $\mathscr{A}^{\sigma}$ the $A$-invariant foliation tangent to $A$-invariant bundle $E^{\sigma}_A$, and $\tilde{\mathscr{A}}^{\sigma}$ the lift of $\mathscr{A}^{\sigma}$, $\sigma = c,u$.

\indent Before coming into the proof of the Main Theorem, we require some more preliminary information.
\begin{enumerate}
\item The foliation structure,
i.e., the semi-conjugacy $h$ maps the center foliation of $f$ to the center foliation of $A$ (Proposition \ref{center foliation}), the unstable foliation of $f$ to the unstable foliation of $A$ (Proposition  \ref{unstable foliation}), and the pre-image of a point on $\TT^2$ is a compact center arc (Proposition \ref{h^{-1}});
\item The smoothness of the unstable bundle, see Theorem \ref{bundle smoothness};
\item The global smoothness of the function which is smooth along foliations, see Theorem \ref{J} and Corollary \ref{Jc}.
\end{enumerate}

Define $d(S,L):=\sup\{d(x,L)|x\in S\}$ as the distance between the set $S$ and $L$, which is used in Proposition 2.5-2.7.

\begin{proposition}\label{bounded distance}
~
\begin{enumerate}
\item If an $\tilde{f}$-invariant foliation $\tilde{\mathscr{F}}$ has bounded distance from linear foliation $\tilde{\mathscr{L}}$, then $\tilde{\mathscr{L}}$ is $A$-invariant.
\item If an $\tilde{f}$-invariant foliation $\tilde{\mathscr{F}}$ has bounded distance from an $A$-invariant linear foliation $\tilde{\mathscr{A}}$, then $\tilde{h}\tilde{\mathscr{F}}(x)=\tilde{\mathscr{A}}(\tilde{h}x)$
\end{enumerate}
\end{proposition}

\begin{proof}
~
\begin{enumerate}
\item
Suppose $\sup_{x\in \RR^2}d(\tilde{\mathscr{F}}(x),\tilde{\mathscr{L}}(x))=M$.\\
From $d(\tilde{h}\tilde{\mathscr{F}}(x),\tilde{\mathscr{L}}(\tilde{h}x))\leq 2||\tilde{h}-id||+M$, we know that
$$d(A\tilde{h}\tilde{\mathscr{F}}(x),A\tilde{\mathscr{L}}(\tilde{h}x))\leq ||A||(2||\tilde{h}-id||+M).$$
and
$$d(A\tilde{h}\tilde{\mathscr{F}}(x),\tilde{\mathscr{L}}(A\tilde{h}x))
=d(\tilde{h}\tilde{\mathscr{F}}(\tilde{f}x),\tilde{\mathscr{L}}(\tilde{h}\tilde{f}x))\leq 2||\tilde{h}-id||+M.$$
So $d(A\tilde{\mathscr{L}}(\tilde{h}x)),\tilde{\mathscr{L}}(A\tilde{h}x))\leq (||A||+1)(2||\tilde{h}-id||+M)<+\infty$. Thus $A\tilde{\mathscr{L}}=\tilde{\mathscr{L}}$.
\item
Suppose $\tilde{h}\tilde{\mathscr{F}}\neq\tilde{\mathscr{A}}$, then there is $x\in\mathbb{T}^2$ such that $d(\tilde{h}\tilde{\mathscr{F}}(x),\tilde{\mathscr{A}}(\tilde{h}x))>0$. Suppose $\sup_{x\in \RR^2}d(\tilde{\mathscr{F}}(x),\tilde{\mathscr{A}}(x))=M$.\\
On the one hand, from the condition that $A$ is expanding, we know that $$\lim_{n\to +\infty}d(A^n\tilde{h}\tilde{\mathscr{F}}(x),A^n\tilde{\mathscr{A}}(\tilde{h}x))=+\infty.$$
On the other hand,
\begin{equation*}
\begin{split}
d(A^n\tilde{h}\tilde{\mathscr{F}}(x),A^n\tilde{\mathscr{A}}(\tilde{h}x))
&=d(\tilde{h}\tilde{\mathscr{F}}(\tilde{f}^nx),\tilde{\mathscr{A}}(A^n\tilde{h}x))\\
&=d(\tilde{h}\tilde{\mathscr{F}}(\tilde{f}^nx),\tilde{\mathscr{A}}(\tilde{h}\tilde{f}^nx))\\
&\leq 2||\tilde{h}-id||+d(\tilde{\mathscr{F}}(\tilde{f}^nx),\tilde{\mathscr{A}}(\tilde{f}^nx))\\
&\leq 2||\tilde{h}-id||+M
\end{split}
\end{equation*}
It is a contradiction. So $\tilde{h}\tilde{\mathscr{F}}=\tilde{\mathscr{A}}$.
\end{enumerate}
\end{proof}

\begin{lemma}\label{P}
\cite[Proposition 4.A.2.]{Potrie}
Given a one dimensional orientable foliation $\mathscr{F}$ of $\TT^2$, we have that there exists a subspace
$L\subset \RR^2$ and $R>0$ such that every leaf of the lift $\tilde{\mathscr{F}}$ of $\mathscr{F}$ lies in a $R$-neighborhood of a translate of $L$.
\end{lemma}

By Lemma \ref{P}, there is a linear foliation $\tilde{L}^u$ such that $\tilde{\mathscr{F}}^u$ has a bounded distance from $\tilde{L}^u$. According to item 1 of Proposition \ref{bounded distance}, $\tilde{L}^u$ must be $A-$invariant. Thus $\tilde{L}^u =\tilde{\mathscr{A}}^c$ or $\tilde{\mathscr{A}}^u$, in particular, $\tilde{L}^u$ has irrational slope. Similarly, there is a linear foliation $\tilde{L}^c$ such that $\tilde{\mathscr{F}}^c$ has a bounded distance from $\tilde{L}^c$, $\tilde{L}^c =\tilde{\mathscr{A}}^c$ or $\tilde{\mathscr{A}}^u$. By item 2 of Proposition \ref{bounded distance}, $\tilde{h}\tilde{\mathscr{F}}^{\sigma}=\tilde{L}^{\sigma}, \sigma=u,c$.

Recall that given two transverse foliations $\mathscr{F}_1$ and $\mathscr{F}_2$ of a manifold $M$, we say that they admit a Global Product Structure if given two points $x,y\in\tilde{M}$, the universal cover of $M$, we have that $\tilde{\mathscr{F}}_1(x)$ and $\tilde{\mathscr{F}}_2(y)$ intersect at a unique point.

\begin{theorem}
\cite[Theorem VIII.2.2.1]{HH}
Consider a codimension one foliation $\mathscr{F}$ without holonomy of a compact manifold $M$. Then, for every $\mathscr{F}^{'}$ foliation transverse to $\mathscr{F}$, we have that $\mathscr{F}$ and $\mathscr{F}^{'}$ have Global Product Structure.
\end{theorem}

Since every leaf of $\mathscr{F}^u$ is embedded copies of $\RR$, it is simply connected, thus $\mathscr{F}^u$ is without holonomy. By the above Theorem, we know that foliation $\mathscr{F}^u$ and its transverse foliation $\mathscr{F}^c$ have Global Product Structure. Thus $\tilde{L}^c\neq \tilde{L}^u$.

\begin{proposition}\label{center foliation}
Under the assumption of Theorem \ref{main}, if $E^c$ is $C^1$, then $\tilde{\mathscr{F}}^c$ has bounded distance from $\tilde{\mathscr{A}}^c$. Thus, $\tilde{h}\tilde{\mathscr{F}}^c(x)=\tilde{\mathscr{A}}^c(\tilde{h}x)$.
\end{proposition}

\begin{proof}
\indent The following method is inspired by \cite{HallHammerlindl}.\\
\indent
Choose a $C^1$ vector field $Y$ closed to $E^c$ on $\mathbb{T}^2$ s.t. $Y(x)\notin E^u(x)$ for any $x\in\mathbb{T}^2$. Consider the $C^1$ neighborhood $\mathscr{U}=\{X\in\mathscr{X}^1(\mathbb{T}^2):X(x)\notin E^u(x),\forall x\in \mathbb{T}^2\}$ of $Y$, then by Peixoto's closing Lemma, there is
a $C^1$ vector field $Y^{\epsilon}\in\mathscr{U}$ s.t. $Y^{\epsilon}$ integrates to $\mathscr{F}^{\epsilon}$ which has a closed leaf on $\mathbb{T}^2$. Thus, by Lemma \ref{P}, the lift $\tilde{\mathscr{F}}^{\epsilon}$ of $\mathscr{F}^{\epsilon}$ has bounded distance $R$ from the rational linear foliation $\tilde{\mathscr{L}}$.\\
\indent Consider the vector field on $\RR^2$. Suppose $\phi_t(x)$ is the flow of the lift $\tilde{Y}^{\epsilon}$ of $Y^{\epsilon}$ on $\RR^2$.
Since $\tilde{f}^{-n}\phi_t(\tilde{f}^{n}x)$ is the flow of the vector field $D\tilde{f}^{-n}(y)\tilde{Y}^{\epsilon}(y)$, there is a flow $\phi^n_t(x)$ s.t. $$\frac{d}{dt}\phi^n_t(x)=X^n(\phi^n_t(x)) \quad\text{and}\quad X^n(x)=\frac{D\tilde{f}^{-n}(\tilde{f}^nx)\tilde{Y}^{\epsilon}(\tilde{f}^nx)}{|D\tilde{f}^{-n}(\tilde{f}^nx)\tilde{Y}^{\epsilon}(\tilde{f}^nx)|}.$$
\indent
Since $f$ is special partially hyperbolic endomorphism with splitting $T\mathbb{T}^2=E^c\oplus E^u$, there is $\lambda\in (0,1)$ such that $||Df|_{E^c(x)}||\leq \lambda||Df|_{E^u(x)}||$ for all $x\in\mathbb{T}^2$.
For any $v=v_c+v_u\notin E^u$, since
$$\frac{|Df^{-n}v_u|}{|Df^{-n}v_c|}\leq \lambda^n\frac{|v_u|}{|v_c|}$$
we have that $<X^n>\rightrightarrows \tilde{E}^c(n\to+\infty)$. So there is subsequence $\{n_k\}$ of $\{n\}$ s.t. $$\sup_{x\in \RR^2}||X^{n_k}(x)-X(x)||\to 0\quad(k\to +\infty)$$ with $|X|=1$ and $X//\tilde{E}^c$.

\indent Fix an initial point $x$ and a period of time $[-T,T]$. Notice that
$$|\phi^n_t(x)-\phi^n_s(x)|\leq \int_s^t|X^n(\phi^n_r(x)|dr\leq |s-t|,\quad s,t\in [-T,T],$$
and
$$|\phi^n_t(x)-x|\leq \int^t_0|X^n(\phi^n_r(x)|dr\leq T,\quad t\in [-T,T].$$
According to Arzela-Ascoli argument, there is a subsequence $\{n'_k\}$ of $\{n_k\}$ such that
$$\phi^{n'_k}_t(x)\rightrightarrows\psi_t(x)\quad(k\to +\infty)$$
Since $\phi^{n'_k}_t(x)$ satisfies that $\phi^{n'_k}_t(x)=x+\int_0^tX^{n'_k}(\phi^{n'_k}_r(x))dr$, and
\begin{equation*}
\begin{split}
&|\int_0^tX^{n'_k}(\phi^{n'_k}_r(x))dr-\int_0^tX(\psi_r(x))dr|\\
&\leq
\int_0^t|X^{n'_k}(\phi^{n'_k}_r(x))-X(\phi^{n'_k}_r(x))|dr+\int_0^t|X(\phi^{n'_k}_r(x))-X(\psi_r(x))|dr\\
&\rightrightarrows 0 \quad (k\to+\infty)
\end{split}
\end{equation*}
Thus $\psi_t(x)=\lim_{k\to+\infty}\phi^{n'_k}_t(x)=x+\int_0^tX(\psi_r(x))dr$. It means that $\psi_t(x)$ is a flow of vector field $X$ on $\RR^2$.
Since $E^c$ is $C^1$, $E^c$ is locally uniquely integrable. Thus $\psi_t(x)\subset \tilde{\mathscr{F}}^c(x)$.\\
\indent Noticing that for any $x\in\mathbb{T}^2$, we have
\begin{equation*}
\begin{split}
d(\tilde{\mathscr{F}}^c(x),\tilde{\mathscr{A}}^c(x))
&\leq ||\tilde{h}-id||+d(\tilde{\mathscr{F}}^c(x),\tilde{\mathscr{A}}^c(\tilde{h}x))\\
&\leq ||\tilde{h}-id||+\limsup_{n\to+\infty}d(\tilde{f}^{-n}\tilde{\mathscr{F}}^{\epsilon}(\tilde{f}^{n}x),A^{-n}\tilde{\mathscr{L}}(A^{n}\tilde{h}x))\\
&\leq 2||\tilde{h}-id||+\limsup_{n\to+\infty}d(\tilde{h}\tilde{f}^{-n}\tilde{\mathscr{F}}^{\epsilon}(\tilde{f}^{n}x),
A^{-n}\tilde{\mathscr{L}}(A^{n}\tilde{h}x))\\
&=2||\tilde{h}-id||+\limsup_{n\to+\infty}d(A^{-n}\tilde{h}\tilde{\mathscr{F}}^{\epsilon}(\tilde{f}^{n}x),
A^{-n}\tilde{\mathscr{L}}(\tilde{h}\tilde{f}^{n}x))\\
&\leq 2||\tilde{h}-id||+\sup_{n\in\NN}d(\tilde{h}\tilde{\mathscr{F}}^{\epsilon}(\tilde{f}^{n}x),
\tilde{\mathscr{L}}(\tilde{h}\tilde{f}^{n}x))\\
&\leq 4||\tilde{h}-id||+\sup_{n\in\NN}d(\tilde{\mathscr{F}}^{\epsilon}(\tilde{f}^{n}x),\tilde{\mathscr{L}}(\tilde{f}^{n}x))
\leq 4||\tilde{h}-id||+R
\end{split}
\end{equation*}
Thus the lift $\tilde{\mathscr{F}}^c$ of $\mathscr{F}^c$ has bounded distance $4||\tilde{h}-id||+R$ from the linear foliation $\tilde{\mathscr{A}}^c$. Furthermore, by item 2 of Proposition \ref{bounded distance}, $\tilde{h}\tilde{\mathscr{F}}^c=\tilde{\mathscr{A}}^c$.
\end{proof}

\begin{proposition}\label{unstable foliation}
Under the assumption of Theorem \ref{main}, if $E^c$ is $C^1$, then $\tilde{\mathscr{F}}^u$ has bounded distance from $\tilde{\mathscr{A}}^u$. Thus, $\tilde{h}\tilde{\mathscr{F}}^u(x)=\tilde{\mathscr{A}}^u(\tilde{h}x)$.
\end{proposition}

\begin{proof}
\indent
Suppose that $\tilde{L}^u=\tilde{\mathscr{A}}^c$, then $\mathscr{F}^u$ and $\mathscr{F}^c$ have a bounded distance from the same linear foliation, so they don't have Global Product Structure. It is a contradiction.
So $\tilde{L}^u\neq \tilde{\mathscr{A}}^c$, then $\tilde{L}^u=\tilde{\mathscr{A}}^u$. Moreover, by item 2 of Proposition \ref{bounded distance}, we get that $\tilde{h}\tilde{\mathscr{F}}^u(x)=\tilde{\mathscr{A}}^u(\tilde{h}x)$.\\
\indent
\end{proof}

Similarly to Lemma 4.8 in \cite{AnderssonRanter}, we have the following proposition.
\begin{proposition}\label{h^{-1}}
Under the assumption of Theorem \ref{main}, if $E^c$ is $C^1$, and $\tilde{\mathscr{F}}^c(x)\neq \tilde{\mathscr{F}}^c(y)$, then $\tilde{h}\tilde{\mathscr{F}}^c(x)\neq \tilde{h}\tilde{\mathscr{F}}^c(y)$. Thus $\tilde{h}^{-1}\tilde{h}(x)\subset \tilde{\mathscr{F}}^c(x)$.
\end{proposition}

\begin{proof}
Denote by $\pi^{u}$ the projection to $E^u_A$ along $E_A^c$, and $\pi^{c}$ the projection to $E_A^c$ along $E^u_A$, where $E^u_A$ and $E_A^c$ are the eigenspaces of $A$, and $E^u_A$ expands stronger than $E_A^c$.\\
\indent Suppose
$\tilde{h}\tilde{\mathscr{F}}^c(x)=\tilde{h}\tilde{\mathscr{F}}^c(y)=\tilde{\mathscr{A}}^c(a)$.
Denote $\tilde{\mathscr{F}}^u(x)\cap\tilde{\mathscr{F}}^c(y)=\{z\}$.\\
\indent On the one hand, since $\tilde{f}^nz\in\tilde{\mathscr{F}}^u(\tilde{f}^nx)$, $$|\pi^c(\tilde{f}^nz-\tilde{f}^nx)|<2\sup_{x\in\RR^2}d(\pi^c\tilde{\mathscr{F}}^u(x),\pi^c\tilde{\mathscr{A}}^u(x)).$$
On the other hand,
$$\tilde{h}\tilde{f}^nz=A^n\tilde{h}z\in A^n\tilde{\mathscr{A}}^c(a)=\tilde{\mathscr{A}}^c(A^na),$$
and
$$\tilde{h}\tilde{f}^nx=A^n\tilde{h}x\in A^n\tilde{\mathscr{A}}^c(a)=\tilde{\mathscr{A}}^c(A^na),$$
thus
$$|\pi^u(\tilde{f}^nz-\tilde{f}^nx)|\leq 2||\tilde{h}-id||+|\pi^u(\tilde{h}\tilde{f}^nz-\tilde{h}\tilde{f}^nx)|=2||\tilde{h}-id||.
$$
Combining the two sides together,
$$\sup_{n\in\NN}d(\tilde{f}^nz,\tilde{f}^nx)\leq 2||\tilde{h}-id||+2\sup_{x\in\RR^2}d(\pi^c\tilde{\mathscr{F}}^u(x),\pi^c\tilde{\mathscr{A}}^u(x))<+\infty.$$
However, by the quasi-isometric property of $\tilde{\mathscr{F}}^u$,
$$ad(\tilde{f}^nz,\tilde{f}^nx)+b\geq d_{\tilde{\mathscr{F}}^u}(\tilde{f}^nz,\tilde{f}^nx)\geq \min_{a\in\mathbb{T}^2} m(Df|_{E^u(a)})^nd_{\tilde{\mathscr{F}}^u}(z,x)\to +\infty(n\to \infty).$$
It is a contradiction, so the conclusion holds.
\end{proof}

A sufficient condition for the $C^2$ unstable bundle is stated in Theorem \ref{bundle smoothness}.

\begin{theorem}\label{bundle smoothness}
Let $f$ be a $C^3$ specially partially hyperbolic endomorphism on $\mathbb{T}^2$.
If $$\inf_{x\in\mathbb{T}^2}\{||Df|_{E^c(x)}||\cdot||Df|_{E^u(x)}||\}>1,$$
then $E^u$ is $C^2$ on $\mathbb{T}^2$.
\end{theorem}
\begin{remark}
The proof of the theorem needs $C^2$ section theorem with fiber contraction $||Df|_{E^c(x)}||\cdot||Df|_{E^u(x)}||^{-1}$ and base contraction $||Df|_{E^c(x)}||$.
\end{remark}
\begin{proof}
Suppose that the diffeomorphism $\tilde{f}:\RR^2\to \RR^2$ is a lift of endomorphism $f:\mathbb{T}^2\to\mathbb{T}^2$ with projection $\pi:\RR^2\to \mathbb{T}^2$.
$\tilde{f}$ has a partially hyperbolic splitting on $\RR^2$, i.e., $T_x\RR^2=\tilde{E}^c(x)\oplus \tilde{E}^u(x)$.
Approximate $T\RR^2=\tilde{E}^c\oplus \tilde{E}^u$ by a $C^{\infty}$ decomposition $T\RR^2=F^c\oplus F^u$. Denote $$D\tilde{f}(x)=\begin{pmatrix}a_{cc}(x)&a_{uc}(x)\\a_{cu}(x)&a_{uu}(x)\end{pmatrix}$$ under the decomposition $T\RR^2=F^c\oplus F^u$. Then $a_{cc}(x)\approx ||D\tilde{f}|_{\tilde{E}^c(x)}||$,
$a_{uu}(x)\approx ||D\tilde{f}|_{\tilde{E}^u(x)}||$, $a_{uc}(x)\approx 0$, $a_{cu}(x)\approx 0$.
Denote by $L(X,Y)$ the linear map from the space $X$ to the space $Y$.
Consider the graph transformation $F:L(F^u,F^c)\to L(F^u,F^c)$, for any $L_x\in L(F^u_{x},F^c_{x})$,
$$ F(L_x)=(a_{cc}(x)L_x+a_{uc}(x))(a_{cu}(x)L_x+a_{uu}(x))^{-1}\in L(F^u_{\tilde{f}x},F^c_{\tilde{f}x}).$$
Denote by $k_x$ the fiber contraction rate, then $k_x\approx||D\tilde{f}|_{\tilde{E}^c(x)}||\cdot||D\tilde{f}|_{\tilde{E}^u(x)}||^{-1}$. On the base, $\lambda_x:=||D\tilde{f}(x)^{-1}||\approx ||D\tilde{f}|_{\tilde{E}^c(x)}||^{-1}$. $E^u$ is $C^2$ if and only if $\tilde{E}^u$ is $C^2$, by the $C^2$ section theorem
\cite{HPS}, which is equivalent to $$\sup_{x\in\RR^2} k_x\lambda_x^2<1,$$ i.e.,$$\inf_{x\in\mathbb{T}^2}\{||Df|_{E^c(x)}||\cdot||Df|_{E^u(x)}||\}>1.$$
\end{proof}

Journe's Theorem describes the global smoothness of functions which are smooth along foliations.

\begin{theorem}\cite{J}(Journe's Theorem)\label{J}
Let $W$ and $V$ be mutually transverse uniformly continuous foliations with $C^r$ leaves on a manifold $M$. If the restriction of $\phi:M\to \RR$ to the leaves of $W$ and $V$ are uniformly $C^r$, then $\phi$ is $C^{r-\epsilon}$ for any $\epsilon>0$.
\end{theorem}

\begin{corollary}\label{Jc}
Let $W_i$ and $V_i$ be mutually transverse uniformly continuous foliations with $C^r$ leaves on a complete manifold $M_i$, $i=1,2$. If $h$ sends $W_1$ to $W_2$ and $V_1$ to $V_2$, the restrictions of $h$ to the leaves of $W_1,V_1$ and their inverse are uniformly $C^r$, then $h$ is $C^{r-\epsilon}$ diffeomorphism for any $\epsilon>0$.
\end{corollary}

\section{Proof of main theorem}
\subsection{Topological conjugacy}
Assume that $A$ is an expanding linear endomorphism with irrational eigenvalue $|\lambda^u_A|>|\lambda^c_A|>1$.
Also assume that $f$ is a $C^r(r\geq 3)$ specially partial hyperbolic endomorphism which is homotopic to $A$.
\begin{proposition}
Under the assumption of Theorem \ref{main}, if the semi-conjugacy $h$ is a topological conjugacy, then $f$ is area-expanding.
\end{proposition}
\begin{proof}
Suppose that $f$ is not area-expanding, then we can claim that there is a $x\in\TT^2$ such that for all $n\in\NN$ we have $|\det(Df^n(x))|\leq 1$, which will be proved in the next paragraph.
Taking the determinant of the equality $$Df^n(x)(e_{E^c_f(x)}e_{E^u_f(x)})=(e_{E^c_f(f^nx)}e_{E^u_f(f^nx)}){\rm diag}\{||Df^n|_{E^c_f(x)}||,||Df^n|_{E^u_f(x)}||\},$$
we can get that
$$||Df^n|_{E^c_f(x)}||\cdot||Df^n|_{E^u_f(x)}||=|\det(Df^n(x))|\frac{\sin\angle(E^c_f(x),E^u_f(x))}{\sin\angle (E^c_f(f^nx),E^u_f(f^nx))}\leq \frac{\sin\angle(E^c_f(x),E^u_f(x))}{\sin\beta}=C,$$
where $\inf_{x\in\mathbb{T}^2}\angle(E^c_f(x),E^u_f(x))=\beta$. Suppose that $\inf_{y\in\TT^2}||Df|_{E^u_y}||\geq\mu>1$, then $$||Df^n|_{E^c_f(x)}||\leq C\mu^{-n}.$$
There is $k\in\NN$ such that $C\mu^{-k}=\lambda\in(0,1)$. Choose $\gamma\in(\lambda, 1)$. There is $\delta>0$ such that for any $y,z\in\TT^2$ with $d(y,z)\leq\delta$ we have $$\frac{\lambda}{\gamma}||Df^k|_{E^c_f(y)}||<||Df^k|_{E^c_f(z)}||<\frac{\gamma}{\lambda}||Df^k|_{E^c_f(y)}||.$$
By induction, for any $y\in\mathscr{F}^c(x)$ with $d_{\mathscr{F}^c}(y,x)<\delta$, any $m\in\NN$, we have
$$d_{\mathscr{F}^c}(f^{km}y,f^{km}x)<C\mu^{-km}\left(\frac{\gamma}{\lambda}\right)^{m}\cdot \delta\leq \gamma^m\delta\to 0\quad (m\to+\infty).$$
Thus $d(f^{km}y,f^{km}x)\to 0\quad (m\to+\infty)$. Since $h$ is continuous,
$$d(A^{km}hy,A^{km}hx)=d(hf^{km}y,hf^{km}x)\to 0\quad (m\to+\infty).$$
It is a contradiction with the expanding property of $A$.\\
\indent Suppose that for any $x\in\TT^2$, there is $n_x\in\NN$, s.t., $|\det(Df^{n_x}(x))|>1$, we will prove that $f$ is area-expanding. There is $\delta_x>0$, $d_x>0$, such that any $y\in B(x,\delta_x)$, $|\det(Df^{n_x}(y))|\geq d_x>1$. Since $\TT^2$ is compact, there are $x_1,\dots, x_m$ such that $\{B(x_i,\delta_{x_i})\}_{i=1}^m$ cover $\TT^2$. Let $N=\max_{1\leq i\leq m}n_{x_i}$, and $\min_{1\leq i\leq m}d_x=\rho^{N}, \rho>1$. For any $z\in\TT^2$, any $n\in\NN$, there is $n_1,\dots, n_k\in [1, N]$ such that $n-n_1-\cdots-n_k<N$, and
$$|\det(Df^{n_{i+1}}(f^{n_1+\cdots+n_i}z))|>\rho^N, 0\leq i\leq k-1.$$
Then $$|\det(Df^n(z))|\geq \rho^{Nk}\cdot\inf_{0\leq j\leq N, w\in\TT^2}|\det(Df^j(w))|>\rho^n\cdot\rho{-N}\inf_{0\leq j\leq N, w\in\TT^2}|\det(Df^j(w))|.$$
\end{proof}

\begin{proposition}\label{small neighborhood}
Under the assumption of Theorem \ref{main}, if $f$ is area-expanding, then $E^u_f$ is $C^2$.
\end{proposition}

\begin{proof}
Suppose that $$|\det(Df^n(x))|=C\rho^n, \rho>1 \quad \text{and}\quad\inf_{x\in\mathbb{T}^2}\angle(E^c_f(x),E^u_f(x))=\beta.$$ There is an integer $n$ such that $C\rho^n\sin\beta> 1$.
Note that $Df^n(x)E_f^u(x)=E_f^u(f^nx)$ for any $x\in \mathbb{T}^2$.
\begin{equation*}
\begin{split}
||Df^n|_{E^c_f(x)}||\cdot||Df^n|_{E^u_f(x)}||&=|\det(Df^n(x))|\frac{\sin\angle(E^c_f(x),E^u_f(x))}{\sin\angle (E^c_f(f^nx),E^u_f(f^nx))}\\
&\geq |\det(Df^n(x))|\sin\beta\\
&\geq C\rho^n\sin\beta> 1
\end{split}
\end{equation*}
Using Theorem \ref{bundle smoothness}, it implies that $E^u_f$ is $C^2$.\\
\end{proof}

If $E^u$ is $C^2$, then the unstable foliation induces a $C^2$ Poincare return map.
Every orientation preserving homeomorphism on $S^1$ with an irrational rotation number is semi-conjugate to the rotation on $S^1$ with the same rotation number. Denjoy theorem guarantees that this semi-conjugacy is, in fact, a conjugacy. Then we can prove that the semi-conjugacy between $f$ and $A$ is invertible.

\begin{proposition}\label{invertible}
Under the assumption of Theorem \ref{main}, if $f$ is area expanding, then the semi-conjugacy $h$ is a topological conjugacy.
\end{proposition}
\begin{proof}
Denote $\Lambda=\overline{\{x:h^{-1}hx=\{x\}\}}$. We claim that if $y\in \Lambda$ then $\mathscr{F}^u(y)\subset \Lambda$, and if $y\notin \Lambda$ then $\mathscr{F}^u(y)\cap \Lambda =\emptyset$.
Before proving the claim, we first finish the proof of the proposition.
Choose a smooth simple closed curve $C$ in $\mathbb{T}^2$ which is transverse to $\mathscr{F}^u$. Consider the Poincare map $R:C\to C$ which is induced by $\mathscr{F}^u$. Since $E^u_f$ is $C^2$ and each $\tilde{\mathscr{F}}^u$ has a bounded distance from an irrational line, $R$ is a $C^2$ circle diffeomorphism without periodic point. Therefore its rotation number is irrational. By Denjoy's theorem, $R$ is conjugate to an irrational rotation, which implies that $\omega(a)=C$ for all $a\in C$. This leads to the conclusion that $\mathscr{F}^u$ is a minimal foliation. Combining this with the previous claim, we deduce that $\Lambda=\mathbb{T}^2$, which means that $h$ is conjugacy.\\
\indent We prove the claim now. If $y$ is not in $\Lambda$, there is $y'$ such that $hy'=hy$.
Thus for any $n\in\NN$, $d(hf^ny,hf^ny')=d(A^nhy,A^nhy')=0$, i.e., $hf^ny=hf^ny'$. Choose $\tilde{y}$, a lift of $y$. We will prove $y'\notin\mathscr{F}^u(y)$. Suppose $y'\in\mathscr{F}^u(y)$. There is a lift $\tilde{y'}\in \tilde{\mathscr{F}}^u(\tilde{y})$ of $y'$. Then $\tilde{h}\tilde{f}^n\tilde{y}=\tilde{h}\tilde{f}^n\tilde{y'}$.
However, when $n\to +\infty$,
$$d(\tilde{h}\tilde{f}^n\tilde{y},\tilde{h}\tilde{f}^n\tilde{y'})
\geq d(\tilde{f}^n\tilde{y},\tilde{f}^n\tilde{y'})-2||\tilde{h}-id||
\geq a^{-1}(d_{\tilde{\mathscr{F}}^u}(\tilde{f}^n\tilde{y},\tilde{f}^n\tilde{y'})-b)-2||\tilde{h}-id||\to +\infty.
$$
It is a contradiction. So $y'\notin\mathscr{F}^u(y)$. For any $z\in\mathscr{F}^u(y)$,
suppose $\tilde{z}\in\tilde{\mathscr{F}}^u(\tilde{y})$ is a lift of $z$, and
$\tilde{y'}$ a lift of $y'$ such that $\tilde{h}\tilde{y'}=\tilde{h}\tilde{y}$.
Denote by $z'$ the projection of $\tilde{z'}=\tilde{\mathscr{F}}^u(\tilde{y'})\cap\tilde{\mathscr{F}}^c(\tilde{z})$.
Since $\tilde{h}\tilde{z'}=\tilde{L}^u(\tilde{h}\tilde{y'})\cap \tilde{L}^c(\tilde{h}\tilde{z})=
\tilde{L}^u(\tilde{h}\tilde{y})\cap \tilde{L}^c(\tilde{h}\tilde{z})=\tilde{h}\tilde{z},
$
$hz'=hz$. Noticing that $z'\in \mathscr{F}^u(y')$, $y'\notin\mathscr{F}^u(y)$, and $z\in\mathscr{F}^u(y)$, so $z\neq z'$. This proves that $z$ is not in $\Lambda$.
\end{proof}

\subsection{Differentiable conjugacy}
In the following Proposition \ref{period equal} and Proposition \ref{equal}, we will prove that under the assumption of Theorem \ref{main}, if $f$ is area expanding and $E^c$ is $C^1$, then $f$ and $A$ have the same center periodic data. Moreover, in Proposition \ref{c-smooth-1}, we will prove that the conjugacy $h$ is $C^1$ along the center foliation. An incidental consequence is that the center foliation is uniformly expanding.

\begin{lemma}\cite[section 6, Table 1]{PSW}
If the tangent bundle $E$ of a foliation $\mathscr{F}$ is $C^1$ , then the leaves and the local holonomies along $\mathscr{F}$ are uniformly $C^1$.
\end{lemma}

\begin{proposition}\label{period equal}
Under the assumption of Theorem \ref{main}, if $f$ is area expanding, and $E^c$ is $C^1$, then for any period points $p$ and $q$, $\lambda^c(p)=\lambda^c(q)$, where $\lambda^c(p)=||Df^{\pi(p)}|_{E^c_f(p)}||^{\frac{1}{\pi(p)}}$, and $\pi(p)$ is the period of $p$.
\end{proposition}
\begin{proof}
Suppose that
$\sup_{p\in Per(f)}\lambda^c(p)=\lambda_+$, $\inf_{p\in Per(f)}\lambda^c(p)=\lambda_-$, and $\lambda_-<\lambda_+$. We will get a contradiction. Denote $\theta = \frac{\ln \lambda^u(A)}{\ln \det(A)}$. Choose a $\epsilon>0$ such that $\lambda_+>\lambda_-e^{\epsilon(1+2\theta^{-1})}$. For $\epsilon>0$, there is a smooth adapted Riemannian metric such that $\lambda_-e^{-\frac{\epsilon}{2}}<||Df|_{E^c_f(x)}||<\lambda_+ e^{\frac{\epsilon}{2}}$ for any $x\in \mathbb{T}^2$. Choose periodic points $p$ and $q$ such that $\lambda^c(p)<\lambda_-e^{\frac{\epsilon}{2}}$ and $\lambda^c(q)>\lambda^+e^{-\frac{\epsilon}{2}}$. Then $$\lambda^c(q)>\lambda^c(p)e^{2\epsilon\theta^{-1}},\quad
\lambda^c(p)e^{-\epsilon}<||Df|_{E^c_f(x)}||<\lambda^c(q)e^{\epsilon}, \quad\forall x\in\mathbb{T}^2.$$
\indent Denote by $\mathscr{A}_L^u(r)$ the unstable leaf of $r$ of size $L$. Since the eigenvalue of $A$ is irrational, there is a constant $C$ such that $$\bigcup_{r\in\ZZ^2}B(\mathscr{A}^u_L(r),CL^{-1})=\RR^2.$$
It implies that for any $k\in\NN$,
$$\bigcup_{r\in A^k\ZZ^2}B(\mathscr{A}^u_{L'}(r),(\det(A))^kCL'^{-1})=\RR^2.$$
\indent Without loss of generality, suppose $f(p)=p$ and $f(q)=q$, otherwise we use $f^{[\pi(p),\pi(q)]}$ instead of $f$. Choose $\pi:\RR^2\to \mathbb{T}^2$, such that $\pi(0)=hq$. Suppose $x$ is a lift of $hp$ and $Ax=x+m$ with $m\in \ZZ^2$.\\
\indent There is $\eta>0$ such that for any $d(r,s)<\eta$ on $\mathbb{T}^2$, $||Df|_{E^c_f(r)}||e^{-\epsilon}<||Df|_{E^c_f(s)}||<||Df|_{E^c_f(r)}||e^{\epsilon}$. And choose $\delta>0$ such that for any $d(r,s)<2\delta$ on $\mathbb{T}^2$ we have that $d(h^{-1}r,h^{-1}s)<\eta$. Choose $L'$ such that $\det(A)^kCL'^{-1}=\sin(\angle(E^c_A,E^u_A))\delta$. There is
$$y=x+m+\cdots+A^{k-1}m+A^kn \quad \text{with} \quad n\in \ZZ^2$$
satisfies
$$d(z,y)<\delta,\quad d(z,0)<L',\quad \text{where}\quad z=\tilde{\mathscr{A}}^u(0)\cap\tilde{\mathscr{A}}^c(y).$$
Suppose $w\in \tilde{\mathscr{A}}^c(0)$ s.t. $d(w,0)=\delta$. Let $v=\tilde{\mathscr{A}}^u(w)\cap  \tilde{\mathscr{A}}^c(y)$, see Figure \ref{picture}.

\begin{figure}[htbp]
  \centering
  \includegraphics[width=10cm]{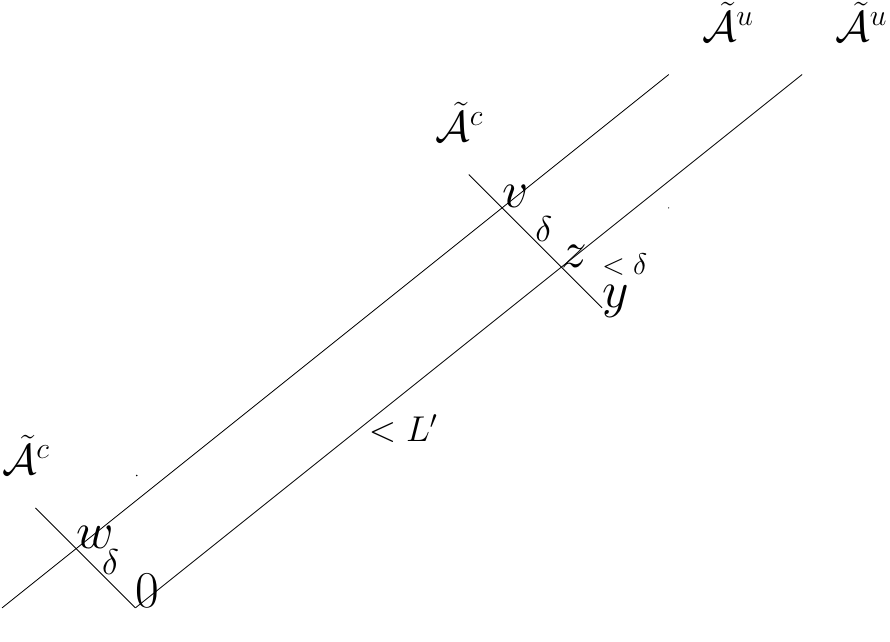}\\
  \caption{Global product on $\RR^2$}\label{picture}
\end{figure}

\indent Denote $\tilde{h}^{-1}y=y'$, $\tilde{h}^{-1}z=z'$, $\tilde{h}^{-1}w=w'$, $\tilde{h}^{-1}v=v'$, $\tilde{h}^{-1}0=o'$. Noticing that $\tilde{h}\circ\tilde{f}=A\circ\tilde{h}$, we have that $$d(\tilde{h}\tilde{f}^{-i}z',\tilde{h}\tilde{f}^{-i}y')=d(A^{-i}z,A^{-i}y)\leq d(z,y)<\delta$$
$$d(\tilde{h}\tilde{f}^{-i}v',\tilde{h}\tilde{f}^{-i}z')=d(A^{-i}v,A^{-i}z)\leq d(v,z)<\delta$$
and
$$d(\tilde{h}\tilde{f}^{-i}w',\tilde{h}\tilde{f}^{-i}o')=d(A^{-i}w,A^{-i}0)\leq d(w,0)=\delta,$$
which imply that center curve connecting $\tilde{f}^{-i}v'$ and $\tilde{f}^{-i}z'$ is contained in $B(\tilde{f}^{-i}y',\eta)$, and $\tilde{f}^{-i}w'$ is contained in $B(\tilde{f}^{-i}o',\eta)$.\\
\indent Note that $y=x+m+\cdots+A^{k-1}m+A^kn$, then $\pi (A^{-i}y)=\pi(x)=hp$ for $0\leq i\leq k$.
Thus for $0\leq i\leq k$,
$$\pi(\tilde{f}^{-i}y')=\pi(\tilde{f}^{-i}\tilde{h}^{-1}y)=\pi(\tilde{h}^{-1}A^{-i}y)=h^{-1}\pi(A^{-i}y)=p.$$
For integer $i\geq 0$,
$$\pi(\tilde{f}^{-i}o')=\pi(\tilde{f}^{-i}\tilde{h}^{-1}0)=\pi(\tilde{h}^{-1}0)=h^{-1}\pi(0)=q.$$
Choose $n\in \NN$ such that $$\frac{\ln L'}{\ln \lambda^u(A)}<n\leq \frac{\ln L'}{\ln \lambda^u(A)}+1.$$
Combined with $\det(A)^kCL'^{-1}=\sin(\angle(E^c_A,E^u_A))\delta$, we get that
$$
\frac{k}{n}\geq \frac{\ln \lambda^u(A)}{\ln \det(A)}\frac{\ln L'+\ln C^{-1}\sin(\angle(E^c_A,E^u_A))\delta}{\ln L'+
\ln \lambda^u(A)}\to \theta (k\to +\infty)
$$
\indent Now we can estimate $d(\pi(\tilde{f}^{-n}v'),\pi(\tilde{f}^{-n}z'))$ and $d(\pi(\tilde{f}^{-n}w'),\pi(\tilde{f}^{-n}o'))$. Suppose that if $r,s$ on $\mathbb{T}^2$ satisfy that $d(r,s)<\delta'$,  then $d(hr,hs)<\delta$. Since $d(\pi v,\pi z)=d(v,z)=d(w,0)=\delta$, $d(h^{-1}\pi v,h^{-1}\pi z)\geq \delta'$.\\
\indent On the one hand, $$d(\tilde{h}\tilde{f}^{-n}o',\tilde{h}\tilde{f}^{-n}z')=d(0,A^{-n}z)<L'\lambda^u(A)^{-n}<1.$$
On the other hand,
\begin{equation*}
\begin{split}
d_{\mathscr{F}^c}(\pi(\tilde{f}^{-n}v'),\pi(\tilde{f}^{-n}z'))&
=\frac{d_{\mathscr{F}^c}(\pi(\tilde{f}^{-n}v'),\pi(\tilde{f}^{-n}z'))}{d_{\mathscr{F}^c}(\pi(\tilde{f}^{-k}v'),\pi(\tilde{f}^{-k}z'))}
\cdot d_{\mathscr{F}^c}(\pi(\tilde{f}^{-k}v'),\pi(\tilde{f}^{-k}z'))\\
&\geq (\lambda^c(q)^{-1}e^{-\epsilon})^{n-k}(\lambda^c(p)^{-1}e^{-\epsilon})^k\cdot d_{\mathscr{F}^c}(h^{-1}\pi v, h^{-1}\pi z)\\
&\geq \delta'\lambda^c(q)^{-n+k}\lambda^c(p)^{-k}e^{-n\epsilon},
\end{split}
\end{equation*}
and
\begin{equation*}
\begin{split}
d_{\mathscr{F}^c}(\pi(\tilde{f}^{-n}w'),\pi(\tilde{f}^{-n}o'))&\leq (\lambda^c(q)^{-1}e^{\epsilon})^n
\cdot d_{\mathscr{F}^c}(h^{-1}\pi w, q)\\
&=\lambda^c(q)^{-n}e^{n\epsilon}\cdot d_{\mathscr{F}^c}(h^{-1}\pi w, q),
\end{split}
\end{equation*}
then
\begin{equation*}
\begin{split}
\frac{d_{\mathscr{F}^c}(\pi(\tilde{f}^{-n}v'),\pi(\tilde{f}^{-n}z'))}{d_{\mathscr{F}^c}(\pi(\tilde{f}^{-n}w'),\pi(\tilde{f}^{-n}o'))}
&\geq e^{-2n\epsilon}\left(\frac{\lambda^c(q)}{\lambda^c(p)}\right)^k\cdot\frac{\delta'}{d_{\mathscr{F}^c}(h^{-1}\pi w, q)}\\
&\geq \left(\frac{\lambda^c(q)}{e^{2(\theta+o(1))^{-1}\epsilon}\lambda^c(p)}\right)^{(\theta+o(1)) n}\frac{\delta'}{d_{\mathscr{F}^c}(h^{-1}\pi w, q)}\to +\infty(n\to +\infty).
\end{split}
\end{equation*}
\qquad It is a contradiction, since the $C^1$ smoothness of $E^u$ implies that the local holonomy along $\mathscr{F}^u$ is uniform $C^1$.
\end{proof}

\begin{proposition}\label{equal}
Under the assumption of Theorem \ref{main}, if $f$ is area expanding, and $E^c$ is $C^1$, then for any period point $p$, $\lambda^c(p)=\lambda^c(A)$.
\end{proposition}
\begin{proof}
From Proposition \ref{period equal}, there is a constant $\lambda$ such that for any period point $p$, $\lambda^c(p)=\lambda$. Livschitz Theorem tells us that there is a H\"older function $\psi:\mathbb{T}^2\to \RR$ such that $$\ln ||Df|_{E^c_f(x)}||=\psi(x)-\psi(fx)+\ln \lambda.$$
Define a new metric $d'$ on $\mathscr{F}^c$ as follows.
$$
d'(x,y):=\int_0^1e^{\psi(\gamma(t))}|\dot{\gamma}(t)|dt,
$$
where $\gamma[0,1]\subset \mathscr{F}^c(x), \gamma(0)=x, \gamma(1)=y$.
Then
$$d'(fx,fy)=\lambda d'(x,y).$$
Suppose that $\psi$ is bounded by $\ln K$ on $\mathbb{T}^2$, then $$K^{-1}d_{\mathscr{F}^c}(x,y)\leq d'(x,y)\leq Kd_{\mathscr{F}^c}(x,y).$$
\indent Suppose that $\tilde{x}$ is a lift of $x$, and $\tilde{y}$ is the lift of $y\in\mathscr{F}^c(x)$ such that $\tilde{y}\in\tilde{\mathscr{F}}^c(\tilde{x})$.\\
\indent On the one hand,
\begin{equation*}
\begin{split}
\lambda^n d'(x,y)&=d'(f^nx,f^ny)\\
&\geq K^{-1}d_{\mathscr{F}^c}(f^nx,f^ny)\\
&=K^{-1}d_{\tilde{\mathscr{F}}^c}(\tilde{f}^n\tilde{x},\tilde{f}^n\tilde{y})\\
&\geq K^{-1}d(\tilde{f}^n\tilde{x},\tilde{f}^n\tilde{y})\\
&\geq K^{-1}(d(\tilde{h}\tilde{f}^n\tilde{x},\tilde{h}\tilde{f}^n\tilde{y})-2||\tilde{h}-id||)\\
&= K^{-1}(\lambda^c(A)^nd(\tilde{h}\tilde{x},\tilde{h}\tilde{y})-2||\tilde{h}-id||)
\end{split}
\end{equation*}
It implies $\lambda\geq \lambda^c(A)$ when $n$ goes to infinity.\\
\indent On the other hand, $\tilde{\mathscr{F}}^c$ is quasi-isometric, thus
\begin{equation*}
\begin{split}
\lambda^n d'(x,y)&=d'(f^nx,f^ny)\\
&\leq Kd_{\mathscr{F}^c}(f^nx,f^ny)\\
&=Kd_{\tilde{\mathscr{F}}^c}(\tilde{f}^n\tilde{x},\tilde{f}^n\tilde{y})\\
&\leq K(ad(\tilde{f}^n\tilde{x},\tilde{f}^n\tilde{y})+b)\\
&\leq K(a(d(\tilde{h}\tilde{f}^n\tilde{x},\tilde{h}\tilde{f}^n\tilde{y})+2||\tilde{h}-id||)+b)\\
&\leq K(a(\lambda^c(A)^nd(\tilde{h}\tilde{x},\tilde{h}\tilde{y})+2||\tilde{h}-id||)+b).
\end{split}
\end{equation*}
It implies $\lambda\leq \lambda^c(A)$ when $n$ goes to infinity.\\
\indent Combining two sides of conclusions, we get that $\lambda=\lambda^c(A)$.
\end{proof}

\begin{corollary}
Under the assumption of Theorem \ref{main}, if $f$ is area expanding, and $E^c$ is $C^1$, then the center foliation is uniformly expanding.
\end{corollary}
\begin{proof}
In the first paragraph of proof of Proposition \ref{period equal}, $\lambda_{-}=\lambda_{+}=\lambda^c(A)>1$.
Then there is a smooth adapted Riemannian metric such that $||Df|_{E^c_f(x)}||>1$ for all $x\in\TT^2$, i.e., the center foliation is uniformly expanding.
\end{proof}

\begin{proposition}\label{c-smooth-1}
Under the assumption of Theorem \ref{main}, if $f$ is area expanding, and $E^c$ is $C^1$, then the conjugacy $h$ is $C^1$ along the center foliation.
\end{proposition}
\begin{proof}
We follow the proof of item 6 based on item 5 in Gan and Shi \cite[Theorem 5.1]{Ganshi}. Since there is no stable foliation, we use the minimal unstable foliation. The holonomy along the unstable leaf between center leaves remains isometric under the metric $d'$, i.e., $|I|=|J|$, where $J=hol^u_{x,y}I$.
The reason is that we can consider more and more small center distance between $I_{-n}$ and $J_{-n}$, where $I_{-n}$ is the nth pre-image of the segment $I$, and $J_{-n}$ is the nth pre-image of the segment $J$. Notice that $d'(fx,fy)=\lambda^c(A)d'(x,y)$, then $$\frac{|I|'}{|J|'}=\lim_{n\to +\infty}\frac{|I_{-n}|'}{|J_{-n}|'}=1.$$
\indent The key observation is that $hx_{\frac{1}{2}}$ is the middle point of $hp$ and $hx$ under $d_{\mathscr{A}^c}$ in $\mathscr{A}^c(hp)$, where $x_{\frac{1}{2}}$ is the middle point of any $p$ and $x\in \mathscr{F}^c(p)$ under metric $d'$ in $\mathscr{F}^c(p)$.
Choose $y_n\in \mathscr{F}^u(p)$ such that $y_n\to x_{\frac{1}{2}}$ when $n\to +\infty$. Denote $z_n=hol^u_{p,y_n}(x_{\frac{1}{2}})$ in $\mathscr{F}^c(y_n)$. Since holonomy $hol^u_{p,y_n}$ along $\mathscr{F}^u(p)$ is isometric, $d'(y_n,z_n)=d'(p,x_{\frac{1}{2}})=d'(x_{
\frac{1}{2}},x)$. Thus $y_n\to x_{\frac{1}{2}}$ implies that $z_n\to x$ when $n\to+\infty$. Since $h\mathscr{F}^{\sigma}=\mathscr{A}^{\sigma} (\sigma = c,u)$ are linear foliations, $$d_{\mathscr{A}^c}(hp,hx_{\frac{1}{2}})=d_{\mathscr{A}^c}(hy_n,hz_n)\to d_{\mathscr{A}^c}(hx_{\frac{1}{2}},hx),\quad n\to +\infty.$$ So $d_{\mathscr{A}^c}(hp,hx_{\frac{1}{2}})=d_{\mathscr{A}^c}(hx_{\frac{1}{2}},hx)$.\\
\indent Finally, by choosing suitable constant $c$, we can use $\psi+c$ instead of $\psi$ to ensure that for any $y\in \mathscr{F}^c(x)$, $d_{\mathscr{A}^c}(hx,hy)=d'(x,y)$, where $\psi$ is the H\"older function in the proof of Proposition \ref{equal}. So
$$\frac{d_{\mathscr{A}^c}(hx,hy)}{d_{\mathscr{F}^c}(x,y)}=\frac{d'(x,y)}{d_{\mathscr{F}^c}(x,y)}=
\frac{\int_0^1e^{\psi(\gamma(t))+c}|\dot{\gamma}(t)|dt}{\int_0^1|\dot{\gamma}(t)|dt}\to e^{\psi(x)+c}\quad (y\to x)$$
\end{proof}

Similarly to Proposition \ref{period equal}, Proposition \ref{equal}, and Proposition \ref{c-smooth-1}, if we exchange $\mathscr{F}^c$ and $\mathscr{F}^u$, noting that the local holonomy along $\mathscr{F}^c$ is uniformly $C^1$ due to the $C^1$-smoothness of $E^c$, then we obtain the corresponding conclusions for the unstable direction.

\begin{proposition}\label{u}
Under the assumption of Theorem \ref{main}, if $f$ is area expanding, and $E^c$ is $C^1$, then the conjugacy $h$ is $C^1$ along the unstable foliation.
\end{proposition}

Since $E^c$ and $E^u$ are both $C^1$, the leaves of $\mathscr{F}^c$ and $\mathscr{F}^u$ are both $C^2$. The following Proposition \ref{c-smooth-2} guarantees that the conjugacy $h$ is $C^2$ along the foliation $\mathscr{F}^c$ and $\mathscr{F}^u$. By Corollary \ref{Jc}, we can know that the conjugacy $h$ is $C^{1+\alpha}$ for $\alpha\in (0,1)$.

\begin{proposition}\label{c-smooth-2}
Under the assumption of Theorem \ref{main}, if $f$ is area expanding, $E^c$ is $C^1$, and the invariant expanding foliation $\mathscr{F}^{\sigma}(\sigma=c,u)$ is uniformly continuous with $C^r$ leaves, then the conjugacy $h$ is $C^r$ along the foliation $\mathscr{F}^{\sigma}$.
\end{proposition}
\begin{proof}
Let $\bar x\in\RR^2$ be the lift of $x\in\mathbb{T}^2$. Then define the dynamical densities $\rho_{x,\bar x}(y)$ with $y\in \mathscr{F}^{\sigma}(x)$
$$\rho_{x,\bar x}(y)=\prod_{i=1}^{+\infty}\frac{||D\tilde{f}|_{\tilde{E}^{\sigma}(\tilde{f}^{-i}\bar x)}||}{||D\tilde{f}|_{\tilde{E}^{\sigma}(\tilde{f}^{-i}\bar y)}||}$$
where $\bar y$ is the lift of $y$ which satisfies $\bar y\in\tilde{\mathscr{F}}^{\sigma}(\bar x)$.\\
\indent $\rho_{x,\bar x}(y)$ is the unique function which satisfies the following requirement.
\begin{enumerate}
\item $\rho_{x,\bar x}(x)=1$, $x\in \mathbb{T}^2$.
\item $\rho_{x,\bar x}(\cdot)$, $x\in \mathbb{T}^2$, are uniformly continuous on $\mathscr{F}_{R}^{\sigma}(x)$ for a fixed $R>0$.
\item $\rho_{fx,\tilde{f}\bar x}(fy)=\frac{||D\tilde{f}|_{\tilde{E}^{\sigma}(\bar x)}||}{||D\tilde{f}|_{\tilde{E}^{\sigma}(\bar y)}||}\rho_{x,\bar x}(y)$ for all $x\in\mathbb{T}^2$ and $y\in \mathscr{F}^{\sigma}(x)$.
\end{enumerate}
Consider $$\tilde{\rho}_{hx,\tilde{h}\bar x}(hy):=\frac{||D\tilde{h}|_{\tilde{E}^{\sigma}(\bar x)}||}{||D\tilde{h}|_{\tilde{E}^{\sigma}(\bar y)}||}\rho_{x,\bar x}(y).$$
Since $h\circ f=A\circ h$, according to Chain rule, $\tilde{\rho}_{hx,\tilde{h}\bar x}(hy)$ satisfies $\tilde{\rho}_{Ahx,A\tilde{h}\bar x}(Ahy)=\tilde{\rho}_{hx,\tilde{h}\bar x}(hy)$. The above three items for $A$ also hold. By uniqueness, $\tilde{\rho}_{hx,\tilde{h}\bar x}(hy)\equiv 1$.
$\rho_{x,\bar x}(y)$ is $C^{r-1}$ about $y$ in $\mathscr{F}^{\sigma}(x)$ since $f$ is $C^r$ and the leaves of foliation are $C^r$. (Referring to \cite[Lemma 4.3]{de la Llave}, and \cite[Lemma 2.3]{G}. We will also give some explanations in the next paragraph.) Combining with
$$||Dh|_{E^{\sigma}(y)}||=||D\tilde{h}|_{\tilde{E}^{\sigma}(\bar y)}||=||D\tilde{h}|_{\tilde{E}^{\sigma}(\bar x)}||\rho_{x,\bar x}(y),$$
we can get that $||Dh|_{E^{\sigma}(y)}||$ is $C^{r-1}$ about $y$ in $\mathscr{F}^{\sigma}(x)$, i.e., $h$ is $C^r$ along the foliaiton.\\
\indent We explain why $\rho_{x,\bar x}(y)$ is $C^{r-1}$ about $y$ in this paragraph.
Choose $\gamma$ such that $\gamma(0)=\bar y,\gamma(t)\subset\tilde{\mathscr{F}}^{\sigma}(\bar x)$, $|\dot{\gamma}(t)|=1$.
Let
$$g(z)=\ln ||D\tilde{f}|_{\tilde{E}(z)}||=\ln \left|{\rm \frac{d}{dt}}\tilde{f}(\gamma(t))\right|,$$ where $z=\gamma(t)$. Denote $u_k(t)=\tilde{f}^{-k}\gamma(t)$.
Suppose $||Df^{-k}|_{E(x)}||\leq C\lambda^k$, $\lambda\in (0,1)$, for all $x\in\TT^2$.
Notice that
$$
{\rm \frac{d}{dt}}u_k(t)=D\tilde{f}^{-k}(\gamma(t))\dot{\gamma}(t)= \prod_{i=0}^{k-1} D\tilde{f}^{-1}(\tilde{f}^{-i}\gamma(t))\cdot \dot{\gamma}(t)
$$
and
$$
\left|\prod_{i=j}^{j+l-1} D\tilde{f}^{-1}(\tilde{f}^{-i}\gamma(t))\right|\leq C\lambda^l,
$$
by using induction for $i$, we can get that
$$\left|{\rm \frac{d^i}{dt^i}}u_{k}(t)\right|\leq C_i\lambda^k(k+1)^{i-1}.$$
More precisely, suppose that the inequality is true when $1\leq i\leq n$, we will prove that the inequality is true when $i=n+1$.
Let $C_i\lambda^k(k+1)^{i-1}\leq D_i$ for all $k\in\NN$. Then for $1\leq j\leq n$,
$$\left|{\rm \frac{d^j}{dt^j}}D\tilde{f}^{-1}(u_i(t))\right|\leq C'\max_{1\leq l\leq n+1}||D^{l}\tilde{f}^{-1}||(\max\{D_1,\dots,D_j\})^j\leq \tilde{C}_n.$$
So
$|{\rm \frac{d^{n+1}}{dt^{n+1}}}u_k(t)|\leq C_{n+1}\lambda^k(k+1)^n$, i.e., the inequality is true for all $i\in\NN$.\\
\indent By Faadi Bruno Formula
$${\rm \frac{d^{r'}}{dt^{r'}}}g\circ u_k(t)=\sum_{b_i\geq 0,\sum ib_i=r',s:=\sum b_i}\frac{r'!}{b_1!\cdots b_{r'}!}g^{(s)}(u_k(t))\prod_{j=1}^{r'} \left(\frac{u_k^{(j)}(t)}{j!}\right)^{b_j}$$
we know that $\left|{\rm \frac{d^{r'}}{dt^{r'}}}g\circ u_k(t)\right|\leq C_{r'}\lambda^k(k+1)^{r'}$, so $\sum_{k=1}^{+\infty}{\rm \frac{d^{r'}}{dt^{r'}}}g\circ u_k(t)$ is uniformly convergent for $r'=1,2,\cdots, r-1$, which implies that $\sum_{k=1}^{+\infty}\left(g\circ u_k(t)-g\circ u_k(0)\right)$ is $C^{r-1}$.
\end{proof}
From the unstable manifold theorem, we know that the leaves of $\mathscr{F}^u$ are $C^r$ if $f$ is $C^r$. Once we prove that the leaves of $\mathscr{F}^c$ are $C^{r-2}$, by Proposition \ref{c-smooth-2} and Corollary \ref{Jc}, we can conclude that the conjugacy $h$ is $C^{r-2-\epsilon}$ for $\epsilon>0$. Moreover, if the leaves of $\mathscr{F}^c$ are $C^{\infty}$, the conjugacy $h$ is $C^{r-3}$ for any $r\in\NN$, thus $h$ is $C^{\infty}$.

\begin{lemma}\label{c1-to-cr}
\cite[Lemma 2.4]{G}
Let $M_1$ and $M_2$ be two one-dimensional complete (not necessary compact) $C^r$ manifolds. Given two $C^r$ expanding diffeomorphisms $g:M_1\to M_1$ and $f:M_2\to M_2$, if $g$ is conjugate to $f$ via a $C^1$ diffeomorphism $h$, then
$h$ is $C^r$.
\end{lemma}

\begin{proposition}
Under the assumption of Theorem \ref{main}, if $f$ is area expanding, and $E^c$ is $C^1$, then the leaves of center foliation of $f$ are $C^{r-2}$.
\end{proposition}

\begin{proof}
We apply the method developed in \cite{G} to the endomorphism on $\mathbb{T}^2$.\\
\indent There is a metric $d^u$ of unstable leaves, under which the holonomy along a center leaf between unstable leaves is isometric. The existence of $d^u$ refers to the first paragraph in the proof of Proposition \ref{c-smooth-1}.
Choose a smooth simple closed curve $C$ in $\mathbb{T}^2$ which is transverse to $\mathscr{F}^u$.
Fix a point $x_0$ in $C$, and see Figure \ref{Phi-def}.
We define the function $\Phi_{x_0}(x)=d^{u}(x,z)$ on $C$, where $z=\mathscr{F}^u(x)\cap\mathscr{F}^c(x_0)$.
Since $\mathscr{F}^c(x_0)$ is $C^2$ and $E^u$ is $C^2$, $\Phi_{x_0}(x)$ is $C^1$ about $x\in C$.
For $x'\in C, x_1\in C$, denote $$z'=\mathscr{F}^u(x')\cap\mathscr{F}^c(x_0), \quad z_1=\mathscr{F}^u(x)\cap\mathscr{F}^c(x_1), \quad
z_1'=\mathscr{F}^u(x')\cap\mathscr{F}^c(x_1).$$
Since $d^u(z,z_1)=d^u(z',z_1')$, $\Phi_{x_0}(x)-\Phi_{x_0}(x')$ is independent of $x_0$. Then we can define continuous function $\phi(x)=\frac{d}{dx}\Phi_{x_0}(x)$, which is independent of $x_0$ in $C$.
\begin{figure}[H]
  \vspace{-2cm}
  \centering
  \begin{minipage}[t]{0.41\linewidth}
    \includegraphics[width=0.9\linewidth]{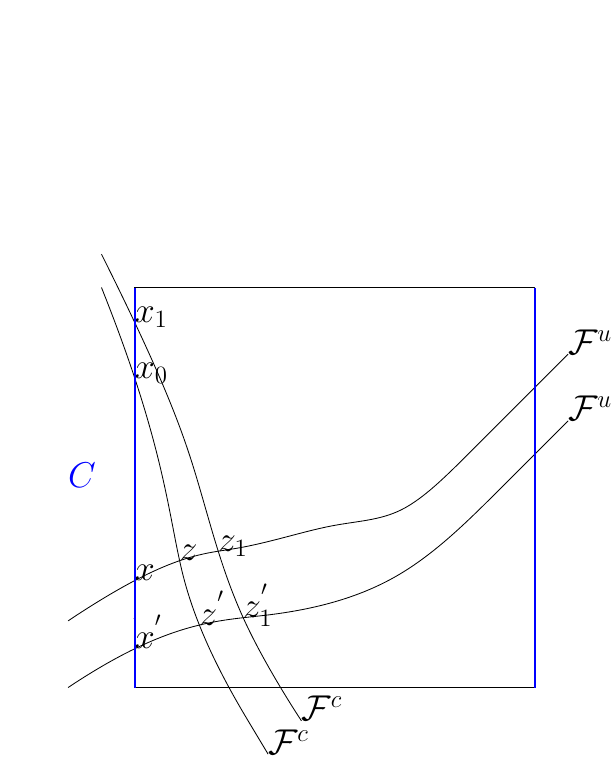}
    \caption{Function $\Phi_{x_0}(x)$}
    \label{Phi-def}
  \end{minipage}
  \hfill
  \begin{minipage}[t]{0.58\linewidth}
    \includegraphics[width=0.9\linewidth]{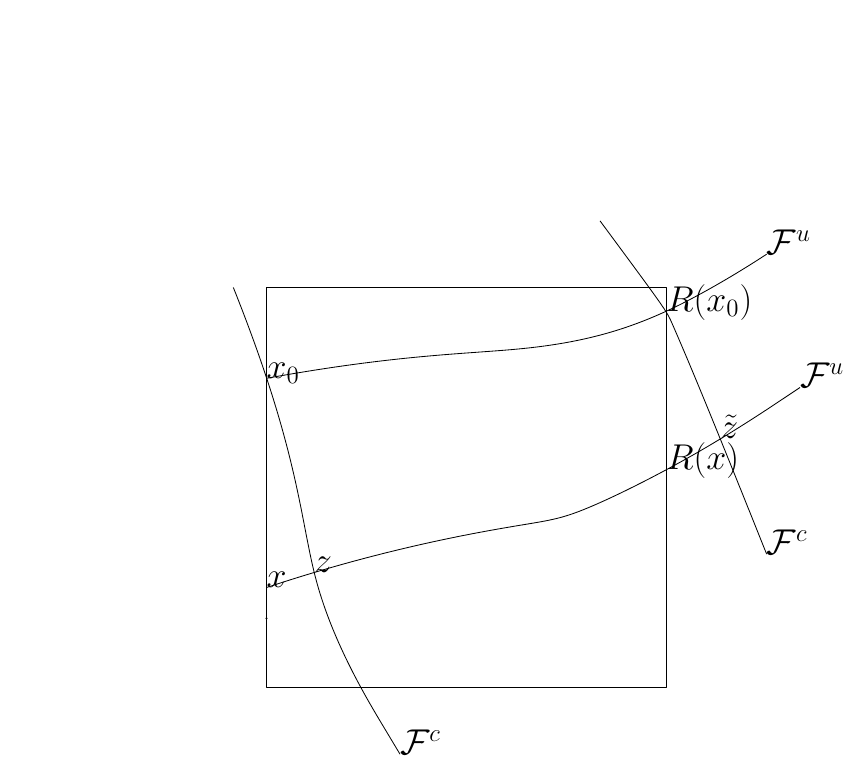}
    \caption{Cohomology relationship}
    \label{Return-map}
  \end{minipage}
\end{figure}
Consider the Poincare map $R:C\to C$ which is induced by $\mathscr{F}^u$.
Observing Figure \ref{Return-map}, we can get the following equation.
\begin{equation*}
\begin{split}
d^u(x_0,R(x_0))&=d^u(z,\tilde{z})\\
&=d^u(x,R(x))+d^u(R(x),\tilde{z})-d^u(x,z)\\
&=d^u(x,R(x))+\Phi_{R(x_0)}(R(x))-\Phi_{x_0}(x),
\end{split}
\end{equation*}
where $z=\mathscr{F}^u(x)\cap\mathscr{F}^c(x_0)$ and $\tilde{z}=\mathscr{F}^u(R(x))\cap\mathscr{F}^c(R(x_0))$.
Take the derivative with respect to $x$ in the above equation. If $R$ is isometric, we can get that $$\phi(R(x))-\phi(x)+B(x)=0,$$ where $B(x)={\rm \frac{d}{dx}}d^u(x,R(x))$. From unstable manifold theorem, we know that $\mathscr{F}^u$ is $C^r$, thus $B(x)$ is $C^{r-1}$.\\
\indent Denote by $\pi^{u}_f$ the holonomy along unstable foliation of $f$, and $r=h^{-1}$. Denote $\bar r=\pi^{u}_f\circ r:C\to C$. Let $R_{\theta}$ be the Poincare map on $C$ induced by $\mathscr{A}^u$, and the rotation number is $\theta$. Since $r\mathscr{A}^u=\mathscr{F}^u$, for $x\in C$, we have
$$R\circ\bar r(x)=\mathscr{F}^u(rx)\cap C=r(\mathscr{A}^u(x))\cap C=\bar r(\mathscr{A}^u(x)\cap C)=\bar r\circ R_{\theta}(x).$$
Consider the metric $(\bar r^{-1})^{*}g(u,v)=g(d\bar r^{-1}u,d\bar r^{-1}v)$ on $C$, where $g$ is standard flat metric on $C$. We have that $R$ is isometric on $(C, (\bar r^{-1})^{*}g)$, since
\begin{align*}
\begin{split}
(\bar r^{-1})^{*}g(dRu,dRv)&=g(d\bar r^{-1}dRu,d\bar r^{-1}dRv)\\
&=g(d\bar r^{-1}dRd\bar r(d\bar r^{-1}u),d\bar r^{-1}dRd\bar r(d\bar r^{-1}v))\\
&=g(dR_{\theta}(d\bar r^{-1}u),dR_{\theta}(d\bar r^{-1}v))\\
&=g(d\bar r^{-1}u,d\bar r^{-1}v)=(\bar r^{-1})^{*}g(u,v).
\end{split}
\end{align*}
\indent
Let $\tilde{C}$ be the lift of $C$ on $\RR^2$. Denote by $\tilde{A}$ the lift of $A$.
Denote by $\tilde{f}$ and $\tilde{r}$ the lift of $f$ and $r$ respectively. Denote $\hat{A}=\pi^{u}_A\circ \tilde{A}:\tilde{C}\to \tilde{C}$ and $\hat{f}=\pi^{u}_{\tilde{f}}\circ \tilde{f}:\tilde{C}\to \tilde{C}$. Denote $\hat{\hat{f}}=\pi^{u}_{\tilde{f}}\circ\tilde{f}:\mathscr{A}^c(0)\to \mathscr{A}^c(0)$.
$\tilde{r}$ induces the quotient map $\hat{r}: \RR^2/\tilde{\mathscr{A}}^u\to  \RR^2/\tilde{\mathscr{F}}^u$, which can be seen as $\hat{r}:\tilde{C}\to \tilde{C}$. Denote $\hat{\hat{r}}=\pi^{u}_{\tilde{f}}\circ\hat{r}:\tilde{C}\to \tilde{\mathscr{A}}^c(0)$.
Since
$\hat{\hat{f}}=\pi^{u}_{\tilde{f}}\circ(\tilde{f}|_{\mathscr{F}^c(\tilde{r}0)})\circ\pi^{u}_{\tilde{f}}$,  $\hat{\hat{f}}$ is conjugate with $\tilde{f}|_{\mathscr{F}^c(\tilde{r}0)}$ by $\pi^{u}_{\tilde{f}}$. By definition, $\hat{r}$ is the lift of $\bar r:C\to C$.
Since $\tilde{r}$ is the conjugacy between $\tilde{A}$ and $\tilde{f}$, $\hat{r}$ is the conjugacy between $\hat{A}$ and $\hat{f}$. $\hat{\hat{r}}$ is the conjugacy between $\hat{A}$ and $\hat{\hat{f}}$.
The relationship between these maps can be seen in Figure 4.
\begin{equation*}
\xymatrix@C=2pc@R=2pc{
&\tilde{\mathscr{A}}^c(0)\circlearrowleft^{\tilde{A}}\ar[r]^{\tilde{r}}\ar[d]
&\tilde{\mathscr{F}}^c(\tilde{r}0)\circlearrowleft^{\tilde{f}}\ar[d]\ar[r]^{\pi^u_{\tilde{f}}}
&\tilde{\mathscr{A}}^c(0)\circlearrowleft^{\hat{\hat{f}}}\\
\tilde{C}\circlearrowleft^{\hat{A}}\ar[rrru]^{\hat{\hat{r}}}\ar@/_/[rrr]_{\hat{r}}\ar[d]
&\RR^2\circlearrowleft^{\tilde{A}}\ar[r]^{\tilde{r}}\ar[l]\ar[d]
&\RR^2\circlearrowleft^{\tilde{f}}\ar[r]\ar[d]
&\tilde{C}\circlearrowleft^{\hat{f}}\ar[u]_{\pi^{u}_{\tilde{f}}}\ar[d]\\
C\ar@/_/[rrr]_{\bar r}\ar[r]&\TT^2\circlearrowleft^{A}\ar[r]^r&\TT^2\circlearrowleft^{f}&C\ar[l]
}
\end{equation*}

\begin{center}
Figure 4: the relationships between maps.
\end{center}
\indent Since $f$ is $C^r$, $\tilde{f}$ and $\mathscr{F}^u$ are $C^r$, furthermore $\pi^u_{\tilde{f}}$ is $C^r$. So $\hat{\hat{f}}$ is $C^r$, and uniformly $C^1$ conjugate with $\tilde{f}|_{\tilde{\mathscr{F}}^c(\tilde{r}0)}$. Thus $\hat{\hat{f}}$ is $C^r$ expanding diffeomorphism. Since $h$ is $C^{1+\alpha}$, $r$ is $C^1$, by definition of $\hat{\hat{r}}$, $\hat{\hat{r}}:\tilde{C}\to \tilde{\mathscr{A}}^c(0)$ is also $C^1$. Notice that $\tilde{C}$ and $\tilde{\mathscr{A}}^c(0)$ are $C^r$ one-dimensional complete manifold, $\hat{A}$ and $\hat{\hat{f}}$ are two $C^r$ expanding diffeomorphisms which are conjugate by $C^1$ diffeomorphism $\hat{\hat{r}}$. By Lemma \ref{c1-to-cr}, $\hat{\hat{r}}$ is $C^r$.
Since $\pi^u_{\tilde{f}}$ is $C^r$, $\hat{r}$ is $C^r$. Thus the base map $\bar r$ of $\hat{r}$ is also $C^r$.\\
\indent Denote $\tilde{\phi}=\phi\circ \bar r$ and $\tilde{B}=B\circ \bar r$. Then $\phi(R(x))-\phi(x)+B(x)=0$ implies that $$
\tilde{\phi}(R_{\theta}(x))-\tilde{\phi}(x)+\tilde{B}(x)=0,
$$
where $\tilde{B}$ is $C^{r-1}$, and $R_{\theta}$ is the Poincare map on $C$ induced by $\mathscr{A}^u$.
$(1,\theta)$ is the eigenvector of $A$, by Liouville's Theorem, so there is a constant $C$ such that
$$|p\theta-q|>\frac{C}{|p|}, \quad \forall p\neq 0\in\ZZ, \quad \forall q\in\ZZ.$$
Thus $\tilde{\phi}$ is $C^{r-3}$ as the continuous solution of cohomological equation \cite[section 2.5]{G}. So $\phi$ is $C^{r-3}$, $\Phi_{x_0}$ is $C^{r-2}$, i.e., the leaves of the center foliation of $f$ are $C^{r-2}$.
\end{proof}

\subsection{Analytic conjugacy}
\begin{lemma}\label{spectral condition}
For $A\in GL(2,\RR)\cap M_2(\ZZ)$ with two irrational eigenvalue $\lambda$ and $\mu$, neither $\lambda=\mu^k$ or $\mu=\lambda^k$ holds when $k\in\NN$.
\end{lemma}
\begin{proof}
Denote $T={\rm tr}(A)$ and $D=\det (A)$. Then the eigenvalues of $A$ are $\frac{1}{2}(T+\sqrt{\Delta})$ and $\frac{1}{2}(T-\sqrt{\Delta})$, where $\Delta= T^2-4D$. Suppose that $k\in\NN$ satisfies
$\frac{1}{2}(T+\sqrt{\Delta})=(\frac{1}{2}(T-\sqrt{\Delta}))^k$ or $\frac{1}{2}(T-\sqrt{\Delta})=(\frac{1}{2}(T+\sqrt{\Delta}))^k$.
Since the eigenvalues of $A$ are irrational, we can get the following equations:
$$\frac{T}{2}=2^{-k}\sum_{i=0}^{[\frac{k}{2}]}C^{2i}_{k}T^{k-2i}\Delta^i$$
and
$$\frac{1}{2}=-2^{-k}\sum_{i=0}^{[\frac{k-1}{2}]}C^{2i+1}_{k}T^{k-1-2i}\Delta^i.$$
The second equation implies that $T<0$ and $k$ is even. However, these means that the left side of the first equation is negative, but the right side of it is positive, which lead to a contradiction.
\end{proof}

\indent For $f$ in the Main theorem, consider $\tilde{f}^{-1}$, the inverse of the lift of $f$. If $f$ is $C^{\omega}$, then by Lemma \ref{spectral condition}, $\tilde{f}^{-1}$ satisfies the assumption of \cite[Theorem 2.2]{CFD}.
For any $p\in \TT^2$, there exist a neighborhood $U$ of $p$ and a homeomorphism $\theta: U\to {\rm Emb^{\omega}}([0,1], \TT^2)$, such that $\theta(x)[0,1]=\mathscr{F}^{\sigma}_{\epsilon}(x)$ and $\theta(x)(0)=x$, where ${\rm Emb^{\omega}}$ is the space of analytic embedding.
In particular, the leaves of $\mathscr{F}^c$ and $\mathscr{F}^u$ are analytic.

\begin{theorem}\label{conjugacy analytic}
Under the assumption of Theorem \ref{main}, $r=\omega$, if $f$ is area expanding, $E^c$ is $C^1$,
and $\mathscr{F}^{\sigma}$ is an expanding foliation with analytic leaves, then $h$ is analytic along $\mathscr{F}^{\sigma}$.
\end{theorem}
\begin{proof}
Suppose that $x$ is a fixed point of $f$, i.e., $f(x)=x$. Choose a kind of parameter $\gamma:\RR\to \mathscr{F}^{\sigma}(x)$ such that $\gamma(0)=x$ and $|\dot{\gamma}(t)|=1$ for all $t\in \RR$. Since $f\mathscr{F}^{\sigma}(x)=\mathscr{F}^{\sigma}(fx)=\mathscr{F}^{\sigma}(x)$, there is an analytic function $\Psi:\RR\to\RR$ such that $f(\gamma(t))=\gamma(\Psi(t))$. Since $f$ is expanding along $\mathscr{F}(x)$, without loss of generality, there is $\alpha\in (0,1)$ such that $|\Psi'(t)|=||Df|_{E^{\sigma}(\gamma(t))}||>\alpha^{-1}>1$. Taking the derivative of $h\gamma\Psi=hf\gamma=Ah\gamma$ at $0$, we have that $\Psi'(0)=\lambda^{\sigma}$. Taking the derivative of
$h\gamma\Psi^n=hf^n\gamma=A^nh\gamma$ at $\Psi^{-n}t$, we have that
$$||Dh|_{E^{\sigma}(\gamma(t))}||=|\prod_{i=1}^n\frac{\lambda^{\sigma}}{\Psi'(\Psi^{-i}(t))}|\cdot||Dh|_{E^{\sigma}(\gamma(\Psi^{-n}t)))}||.$$
Define $$F_n:\RR\to\RR,\quad F_n(t)=\prod_{i=1}^n\frac{\Psi'(\Psi^{-i}(t))}{\lambda^{\sigma}}.$$
\indent For any $t_0\in\RR$, there is an analytic continuation $\Phi$ of $\Psi$,
$$\Phi:\Omega=(-|t_0|-1,|t_0|+1)\times i(-\xi,\xi)\to \Phi(\Omega), \quad |\xi|<\frac{1}{2},$$
such that $|\Phi'(z)|>\alpha^{-1}>1$. Denote
$U= (-|t_0|-\xi,|t_0|+\xi)\times i(-\xi,\xi)$. Since $\Phi^k(k\in\NN)$ is expanding, we have $\Phi^k(\Omega)\supset U$, i.e., $\Omega\supset\Phi^{-k}U$. Then $F_n\in C^{\omega}(U)$ for all $n\in\NN$.\\
\indent We will prove that $F_n$ is uniformly bounded on $U$.
There is a constant $L$ such that $|\ln \Phi'(z_1)-\ln \Phi'(z_2)|\leq L|z_1-z_2|$ for any $z_1,z_2\in U$.
For any $z=x+iy\in U$, we have that
\begin{equation*}
\begin{split}
\ln \left|\frac{F_n(z)}{F_n(x)}\right|&\leq \sum_{i=1}^{+\infty}|\ln \Phi'(\Phi^{-i}z)-\ln \Phi'(\Phi^{-i}x)|\\
&\leq L\sum_{i=1}^{+\infty}|\Phi^{-i}z-\Phi^{-i}x|\\
&\leq \frac{L}{1-\alpha}|z-x|\leq \frac{L\xi}{1-\alpha}.
\end{split}
\end{equation*}
By Livschitz Theorem, there is a H\"older function $\psi:\mathbb{T}^2\to \RR$
$$\ln ||Df|_{E^{\sigma}(y)}||=\psi(y)-\psi(fy)+\ln \lambda^{\sigma}.$$
Thus
$$\ln |\Psi'(t)|=\ln ||Df|_{E^{\sigma}}(\gamma(t))||=\psi(\gamma(t))-\psi(\gamma(\Psi(t)))+\ln \lambda^{\sigma}.$$
It implies that $\ln |F_n(x)|\leq 2||\psi||_{\infty}<+\infty$ for $x\in \RR$.
So $\ln |F_n(z)|\leq 2||\psi||_{\infty}+\frac{L\xi}{1-\alpha}=M<+\infty$ for all $z\in U$.\\
\indent We will also prove that $F_n$ is pointwise convergent on $U$. Since for any $p\in\NN$,
$$\frac{F_{n+p}(z)}{F_n(z)}=\prod_{i=n+1}^{n+p}\frac{\Phi'(\Phi^{-i}z)}{\lambda^{\sigma}}=\exp(\sum_{i=n+1}^{n+p}\ln \Phi'(\Phi^{-i}z)-\ln \Phi'(0))$$
and
$$|\sum_{i=n+1}^{n+p}\ln \Phi'(\Phi^{-i}z)-\ln \Phi'(0)|\leq L\sum_{i=n+1}^{n+p}|\Phi^{-i}z|\leq L\sum_{i=n+1}^{n+p}\alpha^{i}|z|\leq \frac{L}{1-\alpha}|z|\alpha^n\to 0\quad (n\to +\infty),$$
we have that
$$|F_n(z)-F_{n+p}(z)|=|F_n(z)|\cdot |1-\frac{F_{n+p}(z)}{F_n(z)}|\leq e^M\cdot |1-\frac{F_{n+p}(z)}{F_n(z)}|\to 0 $$
when $n$ goes to infinity. So $F_n$ is pointwise convergent to $F$ on $U$.\\
\indent Now we will prove that $F_n$ is uniformly convergent to $F$ on $K$, where $K$ is a compact set in $U$. By Cauchy inequality, since $F_n$ is uniformly bounded on $U$, $F_n'$ is uniformly bounded on $K$, thus $F_n$ is equicontinuous on $K$. For any $\epsilon>0$, there is $\delta>0$, such that for any $n\in\NN$, $|F_n(z_1)-F_n(z_2)|<\frac{1}{3}\epsilon$ for any $|z_1-z_2|\leq \delta$ in $K$. There are $w_1,\dots, w_k$ in $K$ such that $\cup_{1\leq i\leq k} B(w_i,\frac{1}{2}\delta)\supset K$. There is $N$ such that for any $1\leq i\leq k$, $p\in \NN$, and $n>N$, we have $|F_n(w_i)-F_{n+p}(w_i)|<\frac{1}{3}\epsilon$. For any $z\in K$, there is $1\leq j\leq k$, s.t. $z\in B(w_j, \frac{1}{2}\delta)$. For any $n>N$, we have that
$$|F_n(z)-F_{n+p}(z)|\leq |F_n(z)-F_n(w_j)|+|F_n(w_j)-F_{n+p}(w_j)|+|F_{n+p}(w_j)-F_{n+p}(z)|<\epsilon.$$
\indent By Morera's Theorem, we can get that $F$ is analytic on $U$, so $F$ is analytic at $t_0$.
So $||Dh|_{E^{\sigma}(\gamma(t))}||=|F(t)^{-1}|\cdot||Dh|_{E^{\sigma}(\gamma(0))}||$ is analytic.\\
\indent Since $\mathscr{A}^{\sigma}(0)$ is dense in $\TT^2$, $h^{-1}\mathscr{A}^{\sigma}(0)=\mathscr{F}^{\sigma}(h^{-1}0)$ is dense in $\TT^2$. For any $y$ in $\TT^2$, there are $s_n$ such that $\gamma(s_n)\to y$ when $n\to +\infty$.\\
\indent Since $f$ is analytic at $\TT^2$, we claim that $\xi$ can be chosen independently of $t_0$.
Consider homeomorphisms $\theta: U\to {\rm Emb^{\omega}}([0,1], \TT^2)$ and $\theta': fU\to {\rm Emb^{\omega}}([0,1], \TT^2)$. Denote $G_x(t)=\theta'(fx)^{-1}\circ f\circ \theta(x)(t)$, which is analytic about $t$ and varies continuously about $x$.
Denote by $R(x)$ the analytic radius of $G_x(t)$ at $0$.
Denote
$$B(x)=\{y\in U:|(G_x)^{(k)}(0)-(G_y)^{(k)}(0)|\leq 1, \forall k\in\ZZ_{\geq 0}\}.$$
We will prove that $R(y)\geq R(x)$ for any $y\in B(x)$.
For any $\rho>\frac{1}{R(x)}$,
$$\limsup_{k\to +\infty}\sqrt[k]{\frac{|(G_x)^{(k)}(0)|}{k!}}=\frac{1}{R(x)}<\rho,$$
so there exists $K_1$ such that $|(G_x)^{(k)}(0)|<k!\cdot \rho^k$ for any $k>K_1$.
Let $\lceil\frac{1}{\rho}\rceil=N$, then $\rho(N+1)>1$. So
$$k!\cdot\rho^k>N!\rho^N\cdot(\rho(N+1))^{k-N}\to +\infty\quad (k\to +\infty).$$
There exists $K_2>K_1$ such that $k!\cdot\rho^k>1$ for any $k>K_2$. Then
\begin{equation*}
\begin{split}
\frac{1}{R(y)}&=\limsup_{k\to +\infty} \sqrt[k]{\frac{|(G_y)^{(k)}(0)|}{k!}}\\
&\leq \limsup_{k\to +\infty}
\sqrt[k]{\frac{|(G_x)^{(k)}(0)|+1}{k!}}\leq \limsup_{k\to +\infty}\sqrt[k]{\frac{k!\cdot \rho^k+1}{k!}}\\
&\leq \lim_{k\to +\infty} \sqrt[k]{\frac{2k!\cdot \rho^k}{k!}}=\rho
\end{split}
\end{equation*}
Since $\rho$ is arbitrary, we get that $\frac{1}{R(y)}\leq \frac{1}{R(x)}$, i.e., $R(y)\geq R(x)$.
For any $x\in\TT^2$, there exists $B(x)$ such that $R(y)\geq R(x)$. Since $\TT^2$ is compact, there are $x_1,\dots, x_m$ such that $\{B(x_i)\}_{i=1}^m$ cover $\TT^2$. Let $\xi=\min_{1\leq i\leq m}R(x_i)$, then $\inf_{y\in \TT^2}R(y)\geq \xi$. Thus the analytic continuation $\Phi$ of $\Psi$ can be defined on the uniform $\xi$ neighborhood of real line.\\
\indent Notice that
$$||Dh|_{E^{\sigma}(y)}||=\lim_{n\to+\infty}||Dh|_{E^{\sigma}(\gamma(s_n))}||=\lim_{n\to+\infty}|F(s_n)^{-1}|\cdot||Dh|_{E^{\sigma}(\gamma(0))}||.$$
Since $F(s_n+z)$ is analytic when $|z|<\xi$, uniformly bounded, and equicontinuous, by Ascoli Lemma and Morera's Theorem, $||Dh|_{E^{\sigma}(\cdot)}||$ is analytic along $\mathscr{F}^{\sigma}(y)$.
\end{proof}
If $f$ is analytic, by Theorem \ref{conjugacy analytic}, we know that $h$ is analytic along $\mathscr{F}^c$ and $\mathscr{F}^u$. Furthermore, by de la Llave's analytic version of Journe's Lemma \cite{dlL97}, we know that $h$ is analytic.\\
\indent This finishes the proof of the 4th item of Main Theorem.
\section{References}
\begingroup
\renewcommand{\section}[2]{}

\endgroup

\noindent Daohua Yu, School of Mathematical Sciences, Peking University, Beijing 100871, China\\
email: yudh@pku.edu.cn


\begin{thebibliography}{99}
\bibitem{AGGS}
\newblock An, Jinpeng; Gan, Shaobo; Gu, Ruihao; Shi, Yi.
\newblock Rigidity of stable Lyapunov exponents and integrability for Anosov maps.
\newblock Comm. Math. Phys.402(2023), no.3, 2831-2877.

\bibitem{Arteaga}
\newblock Arteaga, Carlos.
\newblock Differentiable conjugacy for expanding maps on the circle.
\newblock Ergodic Theory Dynam. Systems14(1994), no.1, 1-7.

\bibitem{Andersson}
\newblock Andersson, Martin.
\newblock Transitivity of conservative toral endomorphisms.
\newblock Nonlinearity29(2016), no.3, 1047-1055.

\bibitem{AnderssonRanter}
\newblock Andersson, Martin; Ranter, Wagner.
\newblock Partially hyperbolic endomorphisms with expanding linear part.
\newblock Ergodic Theory Dynam. Systems45(2025), no.2, 321-336.

\bibitem{DPU}
\newblock Diaz, Lorenzo J.; Pujals, Enrique R.; Ures, Raul.
\newblock Partial hyperbolicity and robust transitivity.
\newblock Acta Math.183(1999), no.1, 1-43.

\bibitem{de la Llave 0}
\newblock de la Llave, R.
\newblock Invariants for smooth conjugacy of hyperbolic dynamical systems. II.
\newblock Comm. Math. Phys.109(1987), no.3, 369-378.


\bibitem{de la Llave}
\newblock de la Llave, R.
\newblock Smooth conjugacy and S-R-B measures for uniformly and non-uniformly hyperbolic systems.
\newblock Comm. Math. Phys.150(1992), no.2, 289-320.

\bibitem{CFD}
\newblock Cabre, X.; Fontich, E.; de la Llave, Ra.
\newblock The parameterization method for invariant manifolds. I. Manifolds associated to non-resonant subspaces.
\newblock Indiana Univ. Math. J.52(2003), no.2, 283-328.

\bibitem{dlL97}
\newblock de la Llave, R.
\newblock Analytic regularity of solutions of Livsic's cohomology equation and some applications to analytic conjugacy of hyperbolic dynamical systems.
\newblock Ergodic Theory Dynam. Systems 17 (1997), no. 3, 649-662.

\bibitem{Franks}
\newblock Franks, John.
\newblock Anosov diffeomorphisms.
\newblock Global Analysis (Proc. Sympos. Pure Math., Vols. XIV, XV, XVI, Berkeley, Calif., 1968), pp. 61-93
Proc. Sympos. Pure Math., XIV-XVI. American Mathematical Society, Providence, RI, 1970

\bibitem{Gromov}
\newblock Gromov, Mikhael.
\newblock Groups of polynomial growth and expanding maps.
\newblock Inst. Hautes Etudes Sci. Publ. Math.(1981), no.53, 53-73.

\bibitem{Ganshi}
\newblock Gan, Shaobo; Shi, Yi.
\newblock Rigidity of center Lyapunov exponents and su-integrability.
\newblock Comment. Math. Helv.95(2020), no.3, 569-592.

\bibitem{G}
\newblock Gogolev, Andrey.
\newblock Bootstrap for local rigidity of Anosov automorphisms on the 3-torus.
\newblock Comm. Math. Phys.352(2017), no.2, 439-455.

\bibitem{G08}
\newblock Gogolev, Andrey; Guysinsky, Misha.
\newblock $C^1$-differentiable conjugacy of Anosov diffeomorphisms on three dimensional torus.
\newblock Discrete Contin. Dyn. Syst.22(2008), no.1-2, 183-200.

\bibitem{GS}
\newblock Gu, Ruihao; Shi Yi.
\newblock Topological and smooth classification of Anosov maps on torus.
\newblock arXiv:2212.11457 [math.DS]

\bibitem{GX}
\newblock Gu, Ruihao; Xia, Mingyang.
\newblock Semi-conjugacy rigidity for endomorphisms derived from Anosov on the 2-torus.
\newblock arXiv:2311.12669 [math.DS]

\bibitem{HPS}
\newblock Hirsch, Morris W.; Pugh, Charles; Shub, Michael.
\newblock Invariant manifolds.
\newblock Bull. Amer. Math. Soc. 76 (1970), 1015-1019.

\bibitem{HH}
\newblock Hector, Gilbert; Hirsch, Ulrich.
\newblock Introduction to the Geometry of Foliations. Part B. Foliations of Codimension One,
\newblock Second Edition, Aspects of Mathematics, E3, Friedr. Vieweg and Sohn, Braunschweig, 1987.

\bibitem{Heshiwang}
\newblock He, Baolin; Shi, Yi; Wang, Xiaodong.
\newblock Dynamical coherence of specially absolutely partially hyperbolic endomorphisms on $\mathbb{T}^2$.
\newblock Nonlinearity32(2019), no.5, 1695-1704.

\bibitem{HallHammerlindl}
\newblock Hall, Layne; Hammerlindl, Andy.
\newblock Partially hyperbolic surface endomorphisms.
\newblock Ergodic Theory Dynam. Systems41(2021), no.1, 272-282.

\bibitem{HallHammerlindl2}
\newblock Hall, Layne; Hammerlindl, Andy.
\newblock Dynamically incoherent surface endomorphisms.
\newblock J. Dynam. Differential Equations35(2023), no.4, 3651-3663.

\bibitem{coherence}
\newblock Hall, Layne; Hammerlindl, Andy
\newblock Classification of partially hyperbolic surface endomorphisms.
\newblock Geom. Dedicata216(2022), no.3, Paper No. 29, 19 pp.

\bibitem{J}
\newblock Journe, Jean-Lin.
\newblock A regularity lemma for functions of several variables.
\newblock Rev. Mat. Iberoamericana4(1988), no.2, 187-193.

\bibitem{ManePugh}
\newblock Mane, Ricardo; Pugh, Charles.
\newblock Stability of endomorphisms.
\newblock Dynamical systems-Warwick 1974 (Proc. Sympos. Appl. Topology and Dynamical Systems, Univ. Warwick, Coventry, 1973/1974; presented to E. C. Zeeman on his fiftieth birthday), pp. 175-184
Lecture Notes in Math., Vol. 468 Springer-Verlag, Berlin-New York, 1975

\bibitem{Mane}
\newblock Mane, Ricardo.
\newblock Contributions to the stability conjecture.
\newblock Topology 17 (1978), no. 4, 383-396.

\bibitem{M}
\newblock Micena, Fernando.
\newblock Rigidity for some cases of Anosov endomorphisms of torus.
\newblock Discrete Contin. Dyn. Syst.43(2023), no.8, 3082-3097.

\bibitem{MR}
\newblock Cantarino, Marisa; Varao, Regis.
\newblock Anosov endomorphisms on the two-torus: regularity of foliations and rigidity.
\newblock Nonlinearity36(2023), no.10, 5334-5357.

\bibitem{PSW}
\newblock Pugh, Charles; Shub, Michael; Wilkinson, Amie.
\newblock Holder foliations.
\newblock Duke Math. J. 86 (1997), no. 3, 517-546.

\bibitem{Potrie}
\newblock Rafael Potrie Altieri.
\newblock Partial hyperbolicity and attracting regions in 3-dimensional manifolds.
\newblock PhD thesis. arXiv:1207.1822

\bibitem{Shub}
\newblock Shub, Michael.
\newblock Endomorphisms of compact differentiable manifolds.
\newblock Amer. J. Math.91(1969), 175-199.

\bibitem{Shub2}
\newblock Shub, Michael.
\newblock Topologically transitive diffeomorphisms of $\TT^4$.
\newblock In David Chillingworth, editor, Proceedings of the Symposium on Differential Equations and Dynamical Systems, pages 39-40, Berlin, Heidelberg, 1971. Springer Berlin Heidelberg.

\end{thebibliography}
\end{document}